\documentclass[12pt,twoside,a4paper]{article}
\usepackage{amsmath, amsthm,amssymb}
\usepackage[utf8]{inputenc}
\usepackage{float}
\usepackage{sectsty}
\usepackage{tikz}
\usepackage{url}

\usepackage{breakcites}

\sectionfont{\normalsize}
\subsectionfont{\small}
\subsubsectionfont{\small}
\paragraphfont{\small}

\title{The Shark Random Swim \\ (Lévy flight with memory) }
\author{Silvia Businger\footnote{Universität Zürich, silvia.businger@math.uzh.ch}}

\date{}

\newtheorem{theorem}{Theorem}
\newtheorem{definition}{Definition}
\newtheorem{lemma}{Lemma}
\newtheorem{proposition}{Proposition}

\newtheorem*{remark}{Remark}
\def\sgn{\textup{sgn}}
\def\exp{\textup{exp}}

\def\ln{\textup{ln}}

\begin{document}

\maketitle

\begin{abstract}
The Elephant Random Walk (ERW), first introduced by  Schütz and Trimper \cite{und}, is a one-dimensional simple random walk on $ \mathbb{Z} $ having a memory about the whole past. We study the Shark Random Swim, a random walk with memory about the whole past, whose steps are $ \alpha $-stable distributed with  $ \alpha \in (0,2] $. Our aim in this work is to study the impact of the heavy tailed step distributions on the asymptotic behavior of the random walk. We shall see that, as for the ERW, the asymptotic behavior of the Shark Random Swim depends on its memory parameter $ p $, and that a phase transition can be observed at the critical value $ p=\frac{1}{\alpha} $. \\
\textit{keywords}: random walk with memory,  random recursive trees, Yule   processes
\end{abstract}

\section{Introduction}

Anomalous diffusion is a natural  phenomena appearing in physics and biology   for example in porous systems \cite{Porous},  the motion of membrans \cite{membran},  or giving the title to our work, in food searching strategies of marine predators such as sharks \cite{shark}. The Elephant Random Walk (ERW), first introduced by Schütz and Trimper \cite{und}, is a simple model that yields an  anomalous diffusion. The ERW is a one-dimensional simple random walk on $ \mathbb{Z} $ having a memory about the whole past. Specifically, fix a parameter  $ q \in [0,1) $. At each 
 time, the elephant remembers one step from the past chosen uniformly at random. With probability $ q $ the elephant repeats it, and with probability $ 1-q $ it makes a step in the opposite direction.  Equivalently, there is a second formulation of the ERW. For $ q \geq \frac{1}{2} $ set $ p :=2q-1 $. At each step, the elephant either, with probability $ p $, chooses one of the past steps uniformly at random and repeats it, or with probability $ (1-p) $,   decides uniformly at random in which direction it goes. Since its introduction the ERW and similar models have been considered in for example Bercu \cite{Bercu}, Bercu and Laulin \cite{Bercu2},  Boyer and Romo-Cruz \cite{Boyer}, Coletti \textit{et.al.} \cite{Coletti}, da Silva \emph{et. al} \cite{non-Gaussian}, Kürsten \cite{Kursten}, Schütz and Trimper \cite{und},  Serva \cite{Serva},  Wang and \cite{Wang}.

The long time behavior of the ERW depends on the memory parameter $ q $, and has been studied in all regimes of $ q $ by Baur and Bertoin  \cite{Baur}, using the connection of the ERW to urn schemes.  Independently, similar results have been obtained by Coletti \textit{et. al.} \cite{Coletti}. Recently, a different approach using martingales has been studied in Bercu \cite{Bercu}. 

In the subcritical case $ p < 1/2 $ and critical case $ p=1/2$ (equivalently $ q < 3/4 $, respectively $ q=3/4 $),  the re-scaled ERW (with a scaling depending on $ p $) converges in distribution in the Skorohod space to a continuous $ \mathbb{R} $-valued Gaussian process (in the case $ p=0$ the limiting process is a standard Brownian motion). In the supercritical case $ p>1/2 $ (equivalently $ q>3/4)  $ it converges almost surely to the process $ (t^{p}\cdot Z, t \geq 0) $, where $ Z $ is a nondegenerate $ \mathbb{R} $-valued random variable. \\ In this work, we aim to study the long time behavior of a random walk in dimension $ d $ with memory and heavy tailed step distribution, that we will refer to as the Shark Random Swim.

For the ERW the two formulations (using the parameters $ q $ or $ p=2q-1 $) are equivalent, but when the steps are not simple, they yield different processes. Specifically, let $ \xi $ be stable distributed and define a random walk in the following way. At step one, the random walk does a step of size $ \xi$ and  at each step $ n \geq 2 $ the random walk remembers one step from the past chosen uniformly at random, repeats it with probability $ q $, and makes a step in the opposite direction with probability $( 1-q )$.  The process defined in this way then behaves like an ERW multiplied with the random variable $ \xi $.  In order to define a process that behaves different from the ERW, one needs to follow the second formulation of the ERW.

Specifically, let $ \xi_i$, for $ i \in \mathbb{N} $ be i.i.d. $ d $-dimensional  standard isotropic strictly stable random variables with zero shift and stability parameter $ \alpha \in (0,2] $, that is \[\mathbb{E}\left[e^{i  \langle \theta , \xi_1 \rangle} \right]=e^{- \Vert \theta \Vert^\alpha}\quad \textnormal{for} \quad \theta \in \mathbb{R}^d ,\] with $ \Vert \theta \Vert $ denoting the Euclidean norm of $ \theta \in \mathbb{R}^ d $. Now imagine a shark is moving around in the ocean. At time one it is located at position $Y_1= \xi_1 $. At each time $ n  $  it does a step $ Y_{n} $, so its position at time $ n $ is  $S_n= \sum_{i=1}^n Y_i$, in the following way.  With probability $ p \in (0,1) $, it chooses uniform at random one of the past steps $ Y_1,...,Y_{n-1} $ and repeats it, and with probability $ (1-p) $ it does a step independent of the past, that is $ Y_{n}=\xi_{n} $. In contrast to the ERW, the step distributions of the Shark Random Swim are heavy-tailed. 
We are interested in the limiting behavior of $ S_n $ as $ n $ tends to infinity. We shall see that there is a phase transition in the asymptotic behavior of the Shark Random Swim.

In  the subcritical case, that is $ \alpha p < 1 $, we shall see that the random variable $ \left(\frac{1}{n}\right)^{\frac{1}{\alpha}} S_{\lfloor tn \rfloor} $ with $ t \in \mathbb{R}_+ $ converges in distribution to $ t^{\frac{1}{\alpha}} S$, where $ S $ is a $ d $-dimensional isotropic $ \alpha $-stable distributed random variable, and that the scale parameter depends on $ p  $ and the stability parameter $ \alpha $.

We will  see that in supercritical case, that is $ \alpha p >1  $,  for each $ t \in \mathbb{R}_+ $, the random variable $ \frac{1}{n^p} S_{\lfloor tn \rfloor} $ converges to $ t^p V $ in probability, where $V$ is an almost surely finite random variable.

In the critical case, that is $ \alpha p = 1 $, the random variable $(n^t\cdot \log(n))^{-\frac{1}{\alpha}} S_{\lfloor n^t \rfloor}  $ with $ t \in \mathbb{R}_+ $ converges in distribution to $ t^{\frac{1}{\alpha}} \cdot \tilde{S} $, where  $ \tilde{S} $ is a $ d $-dimensional isotropic $ \alpha $-stable distributed random variable with zero shift and scale parameter $ ((1-p)\Gamma(\alpha+1))^{\frac{1}{\alpha}} $.

In the subcritical case  the rescaled Shark Random Swim is expected to converge in distribution in the Skorohod space to a  stable process. For the ERW the limiting process in the supercritical case is Gaussian, and hence can be completely characterized by its covariance functions. As there is no analogue  way to characterize stable processes,  it more effortful to characterize the limiting process of the Shark Random Swim.

In dimension one, we shall see that in the subcritical case, the finite dimensional distributions of the Shark random swim converge to the finite dimensional distributions of an $ \alpha $-stable process, which can be characterized in term of stable integrals. In the critical case the finite dimensional distributions converge to the finite dimensional distributions of a $ \alpha $-stable Lévy process.

The ERW  has a connection to Bernoulli bond percolation on random recursive trees that has first been observed by Kürsten \cite{Kursten}, which still holds true for the Shark Random Swim. Consider a random recursive tree of size $ n $, on which we perform Bernoulli bond percolation. We call the connected components after deleting the edges clusters. We then denote by  $ c_{i,n} $,  $ i \in \mathbb{N} $,  the  cluster rooted at $ i $ and by $ \vert c_{i,n} \vert $ its size. We  shall see that  the position of the Shark Random Swim at time $ n $ can be expressed as $ S_n=\sum_{i} \vert c_{i,n} \vert \xi_i $. The phase transition of the Shark Random Swim is then determined by the asymptotic behavior of the cluster sizes as $ n $ tends to infinity.

In the second section, we shall give a precise description of the connection between the Shark Random Swim and Bernoulli bond percolation, and summarize some results on the latter.  In the first part of the third section, we shall  study the limiting behavior of the Shark Random Swim in the supercritical case.  Our argument will rely on limit results of fragmentation processes of  infinite recursive trees by Baur and Bertoin  \cite{Ornstein}. In the second part of the third section, we shall use fundamental results on Yule processes to study the  limiting behavior of the Shark Random Swim in the subcritical case, and in the third part of the third section  the critical case. Finally, the last part is is devoted to the convergence of the finite dimensional distributions and the characterization of the limiting process.

\label{The shark random swim}

\section{Connection to random recursive trees}
As mentioned in the introduction, it will be crucial for our analysis  to express the position of the Shark Random Swim at time $ n $ in terms of cluster sizes of Bernoulli bond percolation on random recursive trees. 
Recall that random recursive trees are rooted trees with increasing labels along branches that can be build in a recursive manner. We denote by $ T_1 $ the tree with a single node with label $ 1 $. If $ T_{n-1} $ denotes the tree of size $ n-1 $, then  $ T_n $ is build by choosing uniformly at random one of the nodes of $ T_{n-1} $ and adding the $ n $-th node (the node with label $ n $).  We then let  $ 0<p<1 $ and perform Bernoulli bond percolation on the tree, that is each edge is deleted with probability $ (1-p) $, independently of the other edges.

We now consider each step of the shark as adding a node in the random recursive tree. The starting position corresponds to the root, the first step to the node with label $ 1 $, and so on. Following Kürsten \cite{Kursten}, we further add a spin to every node.  The starting position of the random swim corresponds to the root, and we thus assign the spin $ \xi_1 $ to the root.  We then built the tree recursively as described above. For building $ T_n $, we pick one of the nodes of $ T_{n-1} $ and connect the $ n $-th node to the chosen node. With probability $ (1-p) $ the edge connecting the new node to the existing node is deleted, and we assign the spin $ \xi_n $ to the new node. With probability $ p $ the edge is kept, and the new node adopts the spin from the node it is attached to (see Figure 14).
We then get a forest, and we call the trees of the forest clusters.

\begin{figure}[ht]
\label{cluster2}
\begin{center}\parbox{4cm}{\includegraphics[width=3.5cm]{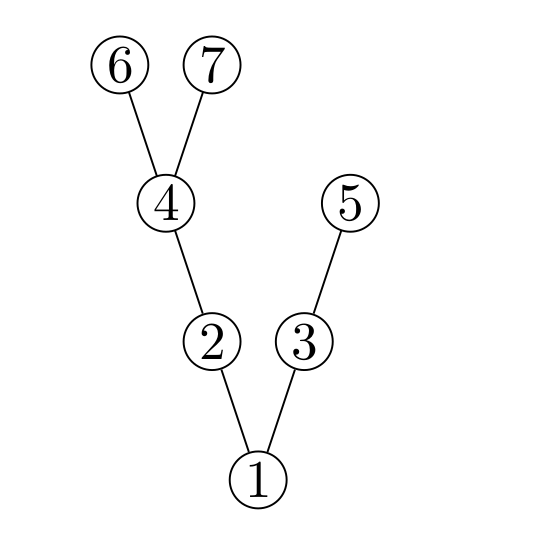}}
\parbox{5cm}{\includegraphics[width=4.5cm]{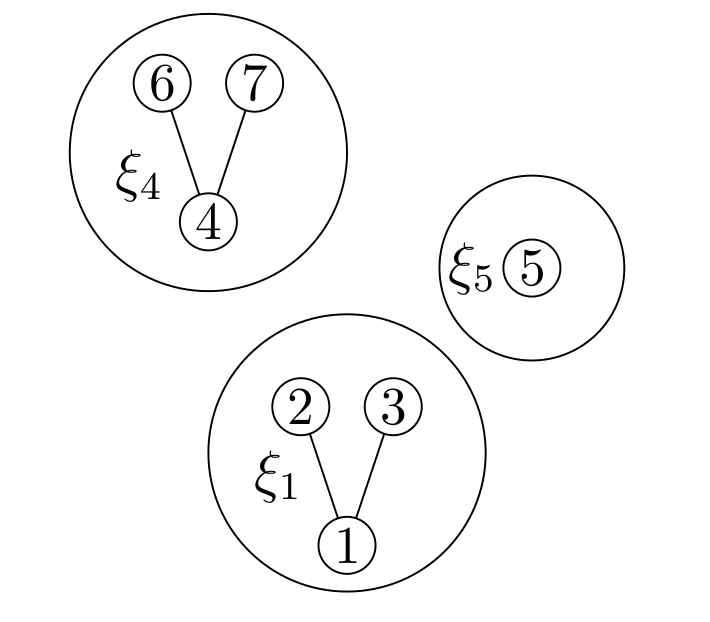}}
\end{center}
\centerline{\bf Figure 14}
\centering{\sl Clusters of the random recursive tree along with the corresponding spin.} 
 \end{figure}  Note that each cluster carries the same spin and different cluster carry different spins. We will denote by $ c_{i,n} $ the cluster rooted at $ i $  and  by $ \vert c_{i,n} \vert $ its size. By convention we have $ c_{i,n}=\emptyset $ and $ \vert c_{i,n} \vert =0 $ if there is no cluster rooted at $ i $. The position of the random swim at time $ n $ can then be rewritten as
\[S_n= \sum_{i} \vert c_{i,n} \vert \xi_i.\]
Note moreover that  we  have $ \sum_i \vert c_{i,n} \vert =n $. 

\subsection{Results on the clusters of Bernoulli bond percolation}

For the limit behavior of the position of the Shark Random Swim at time $ n $, we shall need  to control the asymptotic behavior of the cluster sizes as $ n $ tends to infinity. In this section, we  summarize some results on the cluster sizes of Bernoulli bond percolation. First we mention a connection between the cluster sizes and a Pólya urn scheme. Imagine that we have an urn containing initially $ 1 $ black ball and $ m $ white balls. At each time step we choose a ball uniformly at random from the urn, and return it along with a ball of the same color. Let $ Y_n $ denote the number of black balls after $ n $-draws. It is known:
\begin{lemma}[Pólya]
\label{Polya} Let $ n $ and $ m $ be positive integers. 
\begin{enumerate} 
 \item  Then $ Y_n $ is Beta-binomial distributed with parameters $( n $, $ 1 $, $ m )$, that is to say
 \[\mathbb{P}(Y_n=i) =\left(\begin{matrix}
n \\ i 
\end{matrix} \right) \frac{B(i+1, n-i+m)}{B(1,m)}, \quad i\in\lbrace 1,...,n \rbrace, \]
where $ B $ denotes the Beta function.
 \item The sequence $ \frac{1}{n} Y_n $ converges almost surely  as $ n $ tends to infinity to a random variable $ Y $ that is $ \textup{Beta}(1,m) $ distributed.
 \end{enumerate}
 \end{lemma}
See for example Mahmoud \cite{Mahamoud}, Theorem 3.1 and Theorem 3.2. 
\begin{lemma}
\label{freedman}
Conditioned on $ Y = x $ the random variables $ Y_n $ are  Binomial distributed random variables with parameter $ (n,x) $. 
\end{lemma}
For a proof of this result, see for example Freedman \cite{freedman}.

Now let $ c_{k,n}^\prime $ denote the subtree rooted at node $ k $ after percolation, that is all the nodes and edges that are still connected to the node $ k $, via an increasing sequence of labels, after we performed percolation. The size of this subtree is denoted by $ \vert c_{k,n}^ {\prime} \vert  $, and can be expressed with help of the above urn sheme. We first shall compute how many nodes are in the subtree rooted at node $ k $, \textit{before} the edges are deleted.  Let black balls correspond to nodes that are in the subtree rooted at node $ k $, and white balls correspond to nodes that are not in the subtree rooted at node $ k $. When we build the tree of size $ n $ recursively, once we have arrived at the node $ k $, we have distributed $ k $ nodes and are left to distribute $ n-k $ nodes. In the urn setting this corresponds starting with  $ k $ balls, one black ball (corresponding to the node $ k $) and $ k-1 $ white balls. The number of the remaining $ n-k $ nodes, that will be added to the subtree rooted at node  $ k $,  then corresponds to the number of black balls $ Y(n,k) $  after $ n-k $ draws and thus:
 \begin{equation}
 \label{subtree rooted}
  \vert c_{k,n}^{\prime} \vert \overset{d}{=} \vert c_{1,Y(n,k)} \vert, 
  \end{equation}
where  $Y(n,k) $ is Beta-binomial distributed with parameter $ (n-k,1,k-1) $ and $Y(n,k) =0$ if $ k \geq n $. Now if the edge that connects $ k $ to its parent has been deleted, $ c_{k,n}^\prime $ is a cluster of the Bernoulli bond percolation. Remember that the probability of this event is $ (1-p) $, and that it is independent of the size of the cluster. Thus for all $ i \geq 2 $ we have
\[\mathbb{E}[\vert c_{i,n} \vert]=(1-p) \mathbb{E}[\vert c_{i,n}^\prime \vert].\] 
We will moreover need to control the limiting behavior of the cluster sizes.
 For the root cluster it has been shown (see Kürsten \cite{Kursten}, Section IV):
\begin{lemma} \label{Kursten} We have:
 \[ \mathbb{E} \left[ \vert c_{1,n} \vert \right]= \frac{\Gamma(n+p)}{\Gamma(p+1) \Gamma(n)} \]
and 
\[\mathbb{E} \left[ \vert c_{1,n} \vert^2 \right]= \frac{1}{p}\left(\frac{\Gamma(2p+n)}{\Gamma(2p)\Gamma(n)} - \frac{\Gamma(p+n)}{\Gamma(p)\Gamma(n)} \right).\]
\end{lemma}

 The following three results can be found in Baur and Bertoin  \cite{Ornstein}, Theorem 3.1, Lemma 3.3, and Theorem 3.4. See also Möhle \cite{Mohle}.
\begin{lemma} \label{Ornstein1} The following limit
\[\lim_{n \rightarrow \infty} \frac{1}{n^p}\vert c_{1,n} \vert = X_1 \]
 exists in $ (0, \infty) $ almost surely. Moreover $ X_1 $ is Mittag-Leffler distributed with parameter $ p $. 
\end{lemma}

\begin{proof}
Let $ Y=(Y(t))_{t \in \mathbb{R}_+} $ denote the  Yule process started from $ Y_0=1 $, such that  $ Y(t) $ describes the number of individuals alive at time $ t $, when each individual lives forever and gives birth to children at  rate $ 1 $. It is well know that  $ e^{-t} Y(t)$ is a martingale, and that its terminal value $ W $ exists a.s. and is $\exp( 1 )$ distributed. Thus if we define the birth time of the $ n $-th child $T_n:= \inf \lbrace t: Y(t) = n \rbrace$, we have
\begin{equation}
\label{tuttituuuuuu}
\lim_{n \rightarrow \infty} e^{-T_n} n = W \quad a.s. 
\end{equation}
Now imagine we kill each child with probability $ (1-p) $, independently of the other children. The process of the number of individuals alive at time $ t $, denoted by $ Y^{(p)}(t) $,  is then again a Yule process, with birth rate $ p $, and $ \vert c_{1,n} \vert \overset{d}{=} Y^{p}(T_n) $. Moreover
\begin{equation}
\label{tuttituuuuu2}
\lim_{n \rightarrow \infty} e^{-tp}Y^{(p)}(t)=W^{(p)} \quad a.s,
\end{equation}
where $ W^{(p)} $ is standard exponential distributed. By combining the Equations (\ref{tuttituuuuuu}) and (\ref{tuttituuuuu2}), we arrive at
\begin{equation} \label{nie} \lim_{n \rightarrow \infty} \frac{1}{n^p} Y^{(p)}(T_n)=\frac{W^{(p)}}{W^p} \quad a.s
\end{equation}
Since the left hand side of Equation (\ref{nie}) is independent of the Yule process $ Y $, and hence also of $ W $, we conclude that the limit follows the Mittag-Leffler distribution, by computing its moments.

\end{proof} 

\begin{lemma} \label{Ornstein}
Let $ X_1 $ be defined as in Lemma \ref{Mittag leffler2}. For each $ k \geq 2 $, we have the almost sure convergence 
\[\lim_{n \rightarrow \infty} \frac{1}{n^{ p} } \vert c_{k,n}^{\prime} \vert =  \rho_k \quad a.s, \]
further $ \rho_k $ is equal in distribution to $ X_1 \cdot \beta_k^p $, where $ \beta_k $ is a $ \textup{Beta}(1,k-1) $ distributed random variable independent of $ X_1 $.
\end{lemma} 
\begin{proof}
The result follows from  Lemma \ref{Polya} and the relation $ \vert c_{k,n}^{\prime} \vert \overset{d}{=} \vert c_{1,Y(n,k)} \vert $, where  $Y(n,k) $ is Beta-binomial distributed with parameters $ (n-k,1,k-1) $. 
\end{proof}
Lemma \ref{Ornstein} entails the following statement:
\begin{lemma} \label{Mittag leffler2} For each $ i \geq 2 $ the following limit
\[\lim_{n \rightarrow \infty} \frac{1}{n^p}\vert c_{i,n} \vert = X_i \]
 exists in $ (0, \infty) $ almost surely, and its moments are give by
 \[\mathbb{E}[X_i^q]= (1-p)\frac{\Gamma(q+1)\Gamma(i)}{\Gamma(pq+i)}, \quad q \geq 0.\]  Finally if $ \alpha p >1 $,  the series $\sum_{i=1}^\infty \mathbb{E}[X_i^\alpha] $ converges.
\end{lemma}
Last, we deduce from  Lemma \ref{Kursten} and the asymptotic $  \Gamma(2p+n) \sim \Gamma(n)n^{2p} $ as $ n $ tends to infinity that \begin{eqnarray*} \lim_{n \rightarrow \infty} \frac{1}{n^{2 p} } \mathbb{E} \left[ \vert c_{1,n} \vert^2 \right] & = & \lim_{n \rightarrow \infty}  \frac{1}{n^{2 p} } \cdot \frac{1}{p}\left(\frac{\Gamma(2p+n)}{\Gamma(2p)\Gamma(n)} - \frac{\Gamma(p+n)}{\Gamma(p)\Gamma(n)} \right) \\ &=& \frac{2}{\Gamma(2p+1)}= \mathbb{E}[X_1^2], \end{eqnarray*} since $ X_1 $ is Mittag-Leffler distributed with parameter $ p $, and thus by Scheffe's Lemma 
\[\lim_{n \rightarrow \infty} \frac{1}{n^{p}} \vert c_{1,n} \vert = X_1\]
in $L^2$. A fortiori
\begin{equation} \label{limit expectation2} \lim_{n \rightarrow \infty} \frac{1}{n^{\alpha p} } \mathbb{E} \left[ \vert c_{1,n} \vert^\alpha \right]= \mathbb{E}[X_1^\alpha]
\end{equation}
for each $ \alpha \leq 2 $.

\section{Asymptotic behavior of the Shark Random Swim}
\label{asymptotic shark}
\subsection{Supercritical case $\alpha p > 1$}
\label{supercritical case}
Recall  that the position of the Shark Random Swim at time $ n $ can be expressed as
\[S_n= \sum_{i=1}^n \vert c_{i,n} \vert \xi_i,\]
where  $ \xi_i $, for $ i \in \mathbb{N} $, are i.i.d.  $ d $-dimensional standard isotropic strictly stable random variables with zero shift and stability parameter $ \alpha \in (0,2] $, that is
\[\mathbb{E}\left[e^{i \langle \theta,  \xi_1 \rangle} \right]=e^{- \Vert \theta \Vert^\alpha}, \quad \theta \in \mathbb{R}^d.\]
 Note that if $ \alpha p > 1 $, then $ \alpha \in (1,2] $. The limiting behavior of the  random walk is explained by the limiting behavior of the cluster sizes. Recall that, by Lemma \ref{Mittag leffler2}, for each $ i \geq 1 $ the following limit
\[\lim_{n \rightarrow \infty} \frac{1}{n^p}\vert c_{i,n} \vert = X_i \]
 exists in $ (0, \infty) $ almost surely. We then have:
\begin{theorem} \label{huiii}
Let $ \alpha p > 1 $, let $ t \in \mathbb{R}_+ $, and let $ Z:= \sum_{i=1}^ {\infty} X_i \xi_i  $. Then  $\vert  Z \vert < \infty$  almost surely and
\[ \frac{1}{n^p} S_{\lfloor nt\rfloor}\rightarrow t^p Z  \] 
in probability as $ n $ tends to infinity. 
\end{theorem}
\begin{proof}
We first show that $ \vert Z \vert < \infty $ a.s. Note that, conditionally on $ X_1,...,X_n $, we have
\[\sum_{i=1}^n X_i \xi_i \overset{d}{=} (X_1^\alpha+...+X_n^\alpha)^ {\frac{1}{\alpha} }\xi_1, \]
and we conclude since $ \sum_{i} X_i^\alpha < \infty $ a.s. by Lemma \ref{Mittag leffler2}.
  Next, we aim to show that
\begin{equation} 
\label{convergence in distribution}
\lim_{n \rightarrow \infty} \frac{1}{n^p}S_{\lfloor nt \rfloor}-t^p Z = 0
\end{equation}
in distribution and thus in probability. Let $ \theta \in \mathbb{R}^2 $, we then  have 
\begin{align*}
\exp & \left( i \left\langle \theta,\frac{1}{n^p}S_{\lfloor nt \rfloor}-t^p Z   \right\rangle \right) \\ &=\exp \left( i \left\langle \theta,  \sum_{k=1}^\infty \left(\frac{1}{n^p} \vert c_{k, \lfloor t n \rfloor} \vert 1_{(k \leq \lfloor t n \rfloor)} -t^p X_k \right) \xi_k \right\rangle \right).
\end{align*}
If we let $\mathcal{F}_n:= \sigma(\vert c_{k,1} \vert,...,\vert c_{k,n} \vert, 1 \leq k \leq n)$, then
\begin{align*}
 \mathbb{E} & \left[  \exp \left( i \left\langle \theta, \frac{1}{n^p}S_{\lfloor nt \rfloor}-t^p Z  \right\rangle \right) \bigg\vert \mathcal{F}_\infty \right] \\
&= \exp \left( - \Vert \theta \Vert^\alpha \sum_{k=1}^\infty \bigg\vert \frac{1}{n^p} \vert c_{k, \lfloor t n \rfloor} \vert 1_{(k \leq \lfloor t n \rfloor)} -t^p X_k \bigg\vert^\alpha \right),
\end{align*}
and we are left to show that
\begin{eqnarray} \label{summ} \lim_{n \rightarrow \infty}  \sum_{i=1}^\infty \mathbb{E} \left[ \bigg\vert \frac{1}{n^p}\vert c_{i,n} \vert 1_{( i \leq n )}-X_i \bigg\vert^\alpha \right] = 0. \end{eqnarray}
Indeed, then  \[ \lim_{n \rightarrow \infty} \sum_{i=1}^\infty \bigg\vert \frac{1}{n^p} \vert c_{i, \lfloor t n \rfloor} \vert 1_{(i \leq \lfloor t n \rfloor)} -t^p X_i \bigg\vert^\alpha = 0 \] in probability, and  by dominated convergence we  have
\begin{eqnarray*}
\lim_{n \rightarrow \infty} \mathbb{E} \left[ \exp\left( i  \left\langle \theta, \frac{1}{n^p}S_{\lfloor nt \rfloor}-t^p Z \right\rangle \right)   \right] =  1
\end{eqnarray*}
for every $ \theta \in \mathbb{R}^d $, which proves  (\ref{convergence in distribution}).
To show Equation (\ref{summ}) we first show  \[ \lim_{n \rightarrow \infty}  \mathbb{E} \left[ \bigg\vert \frac{1}{n^p}\vert c_{i,n} \vert -X_i \bigg\vert^\alpha \right] = 0, \] for all $ i \in \mathbb{N} .$ Indeed, recall  that $ \vert c_{k,n}^ {\prime} \vert  $ denotes the size of the subtree rooted at node $ k $, and that it is a cluster if the edge connecting $ k $ to its parent is deleted. Hence
\[ \mathbb{E} \left[ \vert c_{i,n}\vert^{\alpha} \right]= (1-p)\mathbb{E} \left[ \vert c_{i,n}^{\prime} \vert^{\alpha} \right],\]
for all $ i \geq 2 $. We have seen in Equation (\ref{subtree rooted}) that   $ \vert c_{k,n}^\prime \vert  \overset{d}{=} \vert c_{1,Y(n,k)} \vert $, where $Y(n,k)  $ is Beta-Binomial distributed  with parameter $ (n-k,1,k-1) $ and further independent of $ \vert c_{1,i} \vert $, for all $ i \in \mathbb{N} $. By  Lemma \ref{Kursten} we thus have
\begin{eqnarray*}
 \frac{1}{n^{2p}}\mathbb{E} \left[ \vert c_{1,Y(n,k)} \vert^2 \big\vert  Y(n,k)\right] & = & \frac{1}{n^{2p}} \frac{1}{p}\left(\frac{\Gamma(2p+Y(n,k))}{\Gamma(2p)\Gamma(Y(n,k))} - \frac{\Gamma(p+Y(n,k))}{\Gamma(p)\Gamma(Y(n,k))} \right). 
\end{eqnarray*}
By the asymptotics $  \Gamma(2p+n) \sim \Gamma(n)n^{2p} $ and Lemma \ref{Polya} we  have
\[\lim_{n \rightarrow \infty} \frac{1}{n^{2p}}\mathbb{E} \left[ \vert c_{1,Y(n,k)} \vert^2 \big\vert  Y(n,k) \right]   = \frac{2}{\Gamma(2p+1)} \textnormal{B}(k-1)^{2p} \quad a.s,\]
where B($ k-1 $) is a Beta distributed random variable with parameter $ k-1 $.  Since
\[ \frac{1}{n^{2p}}\mathbb{E} \left[ \vert c_{1,Y(n,k)} \vert^2 \big\vert  Y(n,k) \right] \leq \frac{2}{\Gamma(2p+1)}\sup_{n>0} \frac{\Gamma(2p+ n)}{n^{2p}\Gamma(n)} < \infty,\]
we get by dominated convergence  
\begin{eqnarray*}
\lim_{n \rightarrow \infty} \frac{1}{n^{2p}}\mathbb{E} \left[ \vert c_{i,n} \vert^2 \right] &=& \lim_{n \rightarrow \infty} \frac{1}{n^{2p}}(1-p)\mathbb{E} \left[ \vert c^\prime_{i,n} \vert^2 \right] \\ &=& \frac{2(1-p)}{\Gamma(2p+1)}\mathbb{E}[\textup{B}(k-1)^{2p}] = \mathbb{E}[X_i^ 2].
\end{eqnarray*}
We then conclude by Scheffe's Lemma.
 We are thus left to justify the exchange of limit and sum. We shall show that the sum can be bounded by a summable series for all $ n $. We have
\[  \mathbb{E}\left[ \bigg\vert \frac{1}{n^p}\vert c_{i,n} \vert -X_i \bigg\vert^\alpha \right] \leq \frac{1}{n^{\alpha p}} \mathbb{E} [\vert c_{i,n} \vert^\alpha] + \mathbb{E}[X_i^\alpha],\]

and for all $ k \leq n $:
\begin{eqnarray*}
\frac{1}{n^{\alpha p}} \mathbb{E} [\vert c_{k,n} \vert^\alpha] &=& (1-p) \cdot \frac{1}{n^{\alpha p}}\mathbb{E}\left[ \vert c_{k,n}^{\prime} \vert^{\alpha} \right] \\ & = & (1-p) \cdot \frac{1}{n^{\alpha p}} \sum_{i=1}^{n-k} \mathbb{E}\left[ \vert c_{1,i} \vert^{\alpha} \right] \mathbb{P}(Y(n,k) = i).
\end{eqnarray*}
By Equation (\ref{limit expectation2})  there exists an integer $ i_0 $ such that for all $ i \geq i_0 $, we have $ \mathbb{E}\left[ \vert c_{1,i} \vert^{\alpha} \right] \leq c_0 \cdot i^{\alpha p}$, where $c_0:= 2\mathbb{E}[X_1^{\alpha}]$.
Since $ \alpha p < 2 $, we deduce
\begin{eqnarray*} \sum_{i=i_0}^{n-k} \mathbb{E}\left[ \vert c_{1,i} \vert^{\alpha } \right] \mathbb{P}(Y(n,k) = i)  & \leq & c_0 \cdot 
\mathbb{E}[Y(n,k)^{\alpha p}] \\ & \leq & c_0 \cdot \mathbb{E}[Y(n,k)^2]^{\alpha p/2}. \end{eqnarray*}
 Computing the second moment of a Beta-Binomial random variable we  get 
\begin{eqnarray*}
\frac{1}{n^{\alpha p}}\mathbb{E}[Y(n,k)^2]^{\alpha p/2} & = &\frac{1}{n^{\alpha p}} \cdot \frac{((n-k)(2(n-k)+k-1)))^{\alpha p /2}}{(k(k+1))^{\alpha p/2}} \\ & \leq & \frac{3}{k^{\alpha p}}.
\end{eqnarray*}
Using the elementary bound for all $ k \leq n $,
\[\frac{1}{n^{\alpha p}}\sum_{i=1}^{i_0-1} \mathbb{E}\left[ \vert c_{1,i} \vert^{\alpha} \right] \mathbb{P}(Y(n,k) = i) \leq \frac{i_0^{\alpha}}{k^{\alpha p}},\]
 we thus have  for all $ 2 \leq k  $ 
\begin{equation} \label{jetzt} \frac{1}{n^{\alpha p}}\mathbb{E}\left[ \vert c_{k,n}^{\prime} \vert^{\alpha} \right] \leq \frac{3c_0+i_0^{\alpha}}{k^{\alpha p}} 
\end{equation} and

\[  \mathbb{E}\left[ \bigg\vert \frac{1}{n^p}\vert c_{i,n} \vert  -X_i \bigg\vert^\alpha \right] \leq \frac{3c_0+i_0^\alpha}{i^{\alpha p}} + \mathbb{E}[X_i^\alpha],\]
which justifies the interchange of the limit, since   the series $\sum_{i=1}^\infty \mathbb{E}[X_i^\alpha] $ converges for $ \alpha p >1 $, by Lemma \ref{Mittag leffler2}. 

\end{proof}

\subsection{Subcritical case $ \alpha p < 1 $}
\label{erdb}
In this section we shall study the subcritical case. Let $ p<1 $ and us  define the function $f: [0,1-p) \rightarrow \mathbb{R}$, \[ f(x)= \sum_{k=1}^\infty k^{\alpha} \left(1-\left(\frac{x}{1-p}\right)^ p \right)^ {k-1}  \left(\frac{x}{1-p}\right)^p, \] 
that is $ f(x)=\mathbb{E}[G^\alpha] $, where $ G $ is a 
 geometric distributed random variable with parameter $(\frac{x}{1-p})^p  $. Note that $ f$ is Riemann integrable on the interval $ (0,1-p) $, since it is monotone decreasing on $ (0,1-p) $. Then the constant 
\[c(\alpha,p):=  \int_0^{1-p} f(x) dx\]
is finite since  $ f(x) \leq \left( \frac{1-p}{x}\right)^{\alpha p} $ and $ \alpha p < 1 $.
We aim to show:
\begin{theorem}
\label{second case}
Let $ \alpha p <1 $ and let $ t \in \mathbb{R}_+ $. We then have the convergence in distribution \[ \lim_{n \rightarrow \infty}  \left(\frac{1}{n}\right)^{\frac{1}{\alpha}} S_{\lfloor tn \rfloor} = t^{\frac{1}{\alpha}} S,\] where $ S $ is a  $ d $-dimensional isotropic $ \alpha $-stable distributed random variable with zero shift and scale parameter $ c(\alpha,p)^{\frac{1}{\alpha}} $, that is
\[\mathbb{E}\left[e^{i \langle \theta, S \rangle } \ \right]= \exp\left(-c(\alpha,p) \cdot \Vert \theta \Vert^\alpha\right), \quad \theta \in \mathbb{R}^d.\]

\end{theorem}

The proof of Theorem \ref{second case}  will require some  results about Yule processes and Yule processes with mutation. Let $ Y=(Y_t)_{t \in \mathbb{R}_+} $ denote the  Yule process started from $ Y_0=1 $, such that for $ t \in \mathbb{R}_+ $, $ Y_t $ describes the number of individuals alive at time $ t $, when each individual lives forever and gives birth to children at  rate $ 1 $. The following lemma is well-known:

\begin{lemma}
\label{martingale}
The process $ e^{-t} Y(t)$ is a martingale. Its terminal value $ W $ exists a.s. and is $\exp( 1 )$ distributed.
\end{lemma}
Now, assume that we assign a type $ i \in \mathbb{N} $ to each individual. More precisely, assume that the first individual is of type $ 1 $, and each child of an individual either adopts the type of its parent with probability $ p $, or is of a new type (meaning that there is  no individual alive of the same type) with probability $ (1-p) $. This can be understood in the way that each child of our Yule process is either a clone of its parent or a new mutant. 
Let $ Y_i(t) $ denote the population size of individuals of type $ i \in \mathbb{N} $, and let us further introduce the birth time of the first individual of type $ i $, that is $ b_i:=\inf\lbrace t \geq 0: Y_i(t) >0 \rbrace  $. From construction we then have the following lemma.

\begin{lemma}
\label{Yule}
We have
\begin{enumerate}
\item The processes $ (Y_i(t+b_i), t \geq 0) $, $ i \geq 1 $ are i.i.d. Yule processes with birth rate $ p $.
\item As a consequence  $ Y_i(t+b_i) $ is geometric distributed with parameter $ e^{-t p} $, for each $ t \in \mathbb{R}_+ $.
\end{enumerate}
\end{lemma}

To see the connection with the cluster sizes, define the first time when there are $ n $ individuals alive $ T(n) := \inf\lbrace t \in \mathbb{R} : Y(t) = n \rbrace $. 
We then have  

\begin{equation} 
\label{equality in d} 
\vert c_{i,n} \vert \overset{d}{=} Y_i(T(n))  \quad  \textnormal{for all} \quad  n \in \mathbb{N}. 
\end{equation}
This equality in distribution will be useful to prove Theorem \ref{second case}. Recall that if $\mathcal{F}_n= \sigma(\vert c_{i,1} \vert,...,\vert c_{i,n} \vert, 1 \leq i \leq n)$, then
\begin{eqnarray*}
\mathbb{E}\left[\exp \left(i \left\langle \theta,   \left(\frac{1}{n}\right)^{\frac{1}{\alpha}} S_{\lfloor tn \rfloor}  \right\rangle  \right) \bigg\vert \mathcal{F}_n \right]   = \exp\left(-\Vert \theta \Vert^\alpha \frac{1}{n} \sum_{k=1}^{\lfloor tn \rfloor } \vert c_{k,\lfloor tn\rfloor} \vert^\alpha \right).
\end{eqnarray*}
By dominated convergence and Equation (\ref{equality in d}), it will thus  be enough to show the convergence in probability
\begin{equation}
\label{jetzt3}
\lim_{m \rightarrow \infty} \lim_{n \rightarrow \infty} \frac{1}{n}\sum_{i=1}^n  Y_i(T(n))^\alpha \cdot 1_{E_{n,m}} = c(\alpha,p) ,
\end{equation}
where $ E_{n,m} $ is a sequence of events with  \[ \lim_{m \rightarrow \infty} \lim_{n \rightarrow \infty} \mathbb{P}( E_{n,m}) =1.\]
 We  first aim to construct this sequence of events of the form $E_{n,m}= E_{m}(x_n,n) $ with $ x_n=\lfloor \delta n \rfloor  $ and $ 0 < \delta < 1 $. Let  $ \textup{Geom}(r) $ denote a generic geometric distributed random variable with parameter $ r $. We then have the following result.
\begin{lemma}
\label{sequuence}
Let
 $(x_n)_{n \in \mathbb{N}} $ be a sequence with $ 0 < x_n < n $ and $\lim_{n \rightarrow \infty} x_n = \infty $. Then there  exists  a positive sequence  $ \varepsilon_m $  with $ \varepsilon_m \downarrow 0 $ as $ m $ tends to infinity, and  a sequence of events $ E_{m}(x_n,n) $ for which \[ \lim_{m \rightarrow \infty} \lim_{n \rightarrow \infty} \mathbb{P}(E_{m}(x_n,n))=1,\] such that on $  E_{m}(x_n,n) $ we have the bounds
\[\underline{X}_i(n, \varepsilon_m) \leq Y_i(T(n)) \leq \overline{X}_i(n, \varepsilon_m), \quad \textnormal{for all}  \quad i \in \lbrace x_n,...,n \rbrace,\]
where  $  \overline{X}_i(n, \varepsilon_m) $   are independent  random variables with  \[ \overline{X}_i(n, \varepsilon_m) \overset{d}{=} \textup{Geom}\left( \left( \frac{i-1}{n(1-p)(1+\varepsilon_m)}\right)^p \right), \quad \textnormal{for} \quad i-1 \leq n(1-p)(1+\varepsilon_m), \]  and $ \overline{X}_i(n, \varepsilon_m)=0$, for $ i-1 > n(1-p)(1+\varepsilon_m).  $ Similarly  $\underline{X}_i(n, \varepsilon_m)  $ are  independent random variables with \[\underline{X}_i(n, \varepsilon_m) \overset{d}{=}  \textup{Geom}\left( \left( \frac{i+1}{n(1-p)(1-\varepsilon_m)}\right)^p \right),\quad \textnormal{for} \quad i+1 \leq n(1-p)(1-\varepsilon_m),\]
and $\underline{X}_i(n, \varepsilon_m)=0  $ for $ i+1 > n(1-p)(1-\varepsilon_m) $.

\end{lemma}
\begin{proof}
In order to apply Lemma \ref{Yule}, we aim to find a deterministic upper and lower bound for $ T(n)-b_i $. Let $0< \mu_m<1 $ be a sequence with $ \mu_m \downarrow 0 $, and define the sequence of events   \[ E_{k,m}^1:=\lbrace \omega \in \Omega :  W(\omega)(1-\mu_m) \leq e^{-T(k)(\omega)} k \leq W(\omega)(1+\mu_m) \rbrace. \]  By Lemma \ref{martingale} we have  $ \lim_{m \rightarrow \infty} \lim_{n \rightarrow \infty} \mathbb{P}(\bigcap_{k=n} E^1_{k,m})=1 $. Note that  on $  E^1_{k,m} $ we have
\[ \ln(k)- \ln(W) - \ln(1+\mu_m) \leq T(k)\leq \ln(k)- \ln(W) - \ln(1-\mu_m).  \]
 
Now let $ D(k) $ denote the number of different types that can be observed at time $ T(k) $. Since at time $ T(k) $ there are exactly $ k $ individuals alive, $ D(k) $ is  equal in distribution to $ \sum_{i=1}^k \textup{Ber}_i(1-p) $, where $ \textup{Ber}_i(1-p) $ are i.i.d. Bernoulli random variables. Let 
\[E^2_{k,m}:= \lbrace \omega  \in \Omega :k(1-p)(1-\mu_m) \leq D(k)(\omega)\leq k(1-p)(1+\mu_m) \rbrace, \]
by the law of large numbers $\lim_{m \rightarrow \infty}  \lim_{n \rightarrow \infty} \mathbb{P}(\bigcap_{k=n} E^2_{k,m})=1  $. On the event $ E^2_{k,m} $ we have  \[ b_{\lceil k(1-p)(1+\mu_m) \rceil} \geq T(k) ,\] and we conclude that on the event $  E^1_{k,m} \cap E^2_{k,m} $:
\begin{eqnarray*} b_k & \geq & T\left( \left\lfloor \frac{k}{(1-p)(1 + \mu_m)}  \right\rfloor \right) \\ & \geq & \ln(k-1)- \ln(W) - \ln(1-p)-2\ln(1+ \mu_m) .
 \end{eqnarray*}
 Now, define the event  $ E_{m}(x_n,n):= \bigcap_{k=x_n} (E_{k,m}^1 \cap E_{k,m}^2) $, and note that 
 \[\lim_{m \rightarrow \infty} \lim_{n \rightarrow \infty} \mathbb{P}(E_{m}(x_n,n))=1.\] For each $  i \in \lbrace x_n,...,n\rbrace $ we have on $ E_{m}(x_n,n) $ the inequality 
 \[T(n)-b_i  \leq \ln(n)-\ln(i-1)+\ln(1-p)+2\ln(1+\mu_m)-\ln(1-\mu_m).\]
If $ i -1 > n(1-p)(1+\varepsilon)  $ we have $ T(n)-b_i < 0 $ and hence at time $ T(n) $ there is no individual of type $ i $ alive, that is  $ Y_i(T(n))=0 $. For $ i -1 \leq n(1-p)(1+\varepsilon) $ we can  define 
  \[ \overline{X}_i(n, \varepsilon):=  Y_i(\ln(n)-\ln(i-1)+\ln(1-p)+\ln(1 + \varepsilon) +b_i), \]
  and choosing the sequence $ \varepsilon_m $ such that \[ \ln(1+\varepsilon_m) \geq 2\ln(1+\mu_m)-\ln(1-\mu_m),\] that is $ \varepsilon_m \geq\frac{1}{(1-\mu_m)} (\mu_m^2+3\mu_m)$, we then arrive at
   \[ Y_i(T(n)) \leq \overline{X}_i(n,\varepsilon_m) \quad \textnormal{for} \quad i \in \lbrace  x_n,...,n\rbrace.  \]
In the same spirit one can show the lower bound. 
 
\end{proof}
We are thus left to show:
\begin{lemma}
\label{expectations second moment}
Let $ 0 < \delta < 1 $ and define $ E_{n,m}:=E_m(\lfloor \delta n \rfloor,n) $. We have the convergence in probability \[\lim_{m \rightarrow \infty} \lim_{n \rightarrow \infty} \frac{1}{n} \sum_{i=1}^n  Y_i(T(n))^{\alpha} \cdot 1_{E_{n,m}}  = c(\alpha,p). \] 

\end{lemma}
\begin{proof}
We aim to use the second moment method. We first show that 
\begin{equation}
\label{jetzt2}
\lim_{m \rightarrow \infty}  \lim_{n \rightarrow \infty} \frac{1}{n} \sum_{i=1}^n \mathbb{E} \left[  \vert Y_i(T(n))^{\alpha} \cdot 1_{E_{n,m}} \right]  = c(\alpha,p).
\end{equation}
Recall that by inequality (\ref{jetzt}), in the proof of Theorem \ref{huiii}, we have for all $ 2 \leq k  $:
\begin{equation*} \frac{1}{n^{\alpha p}}\mathbb{E}\left[ \vert c_{k,n} \vert^{\alpha} \right] = \frac{1}{n^{\alpha p}}(1-p)\mathbb{E}\left[ \vert c_{k,n}^{\prime} \vert^{\alpha} \right] \leq (1-p) \frac{3c_0+i_0^{\alpha}}{k^{\alpha p}} ,
\end{equation*} 
where $ i_0 $ and $ c_0 $ are constants and thus 
\[ \frac{1}{n} \sum_{i=1}^{\lfloor \delta n \rfloor} \mathbb{E}[Y_i(T(n))^\alpha \cdot 1_{E_{n,m}}]\leq \frac{1}{n}\mathbb{E}[\vert c_{1,n} \vert^\alpha]+ \frac{1-p}{n^{1-\alpha p}} (3c_0+ i_0^\alpha) \sum_{i=2}^{\lfloor \delta n \rfloor} \frac{1}{i^{\alpha p}}.\]
 We recognize the partial sum of an $ \alpha p  $-series, which can be bounded by 
\[\frac{1-p}{n^{1-\alpha p}} \sum_{i=2}^{\lfloor \delta n \rfloor} \frac{3c_0 + i_0^\alpha}{i^{\alpha p}} \leq \frac{1-p}{n^{1-\alpha p}} \cdot 
 (3c_0  + i_0^\alpha )\left(1 + \frac{(n \delta)^{1-\alpha p}-1}{1-\alpha p}\right), \]
as can be found in Chlebus  \cite{harmonic}, and we arrive at:
\[\limsup_{n \rightarrow \infty} \frac{1}{n} \sum_{i=1}^{\lceil \delta n \rceil} \mathbb{E}[Y_i(T(n))^\alpha \cdot 1_{E_{n,m}}] \leq ( 3c_0 + i_0^\alpha )(1-p) \cdot \frac{\delta^{1-\alpha p}}{1-\alpha p}.\]
Now let  $ f_\varepsilon(x):=f \left( \frac{x}{1+\varepsilon} \right) $, and note that $ f_\varepsilon $ is  Riemann integrable on the interval $ (0,1) $. We  then have 
\begin{eqnarray*}
\frac{1}{n}  \sum_{i= \lfloor \delta n \rfloor + 1}^n \mathbb{E}[\overline{X}_i(n , \varepsilon_m)^\alpha]= \frac{1}{n} \sum_{i=\lfloor \delta n \rfloor }^{\lfloor n (1-p) \rfloor} f_{\varepsilon_m}\left(\frac{i}{n}\right) ,
\end{eqnarray*}
and thus by Lemma \ref{sequuence}
\[ \limsup_{n \rightarrow \infty} \frac{1}{n} \sum_{i=\lfloor \delta n \rfloor +1}^n \mathbb{E}[ Y_i(T(n))^{\alpha} \cdot 1_{E_{n,m}} ] \leq  \int_\delta^{1-p} f_{\varepsilon_m}(x) dx. \]
In the same spirit one can show that
\[\liminf_{n \rightarrow \infty} \sum_{i=\lfloor \delta n \rfloor +1 }^n \mathbb{E}[Y_i(T(n))^{\alpha} \cdot 1_{E_{n,m}}] \geq  \int_\delta^{1-p} g_{\varepsilon_m}(x) dx,\]
where  $g_\varepsilon(x):=f\left(\frac{x}{1-\varepsilon} \right) $.
Letting $ \delta $ tend to zero we derive that 
\begin{align*} 
\int_0^{1-p} g_{\varepsilon_m}(x) dx &\leq  \liminf_{n \rightarrow \infty} \frac{1}{n} \sum_{i=1}^n \mathbb{E}[ Y_i(T(n))^\alpha \cdot 1_{E_{n,m}}] \\ & \leq  \limsup_{n \rightarrow \infty} \frac{1}{n} \sum_{i=1}^n \mathbb{E}[ Y_i(T(n))^\alpha \cdot 1_{E_{n,m}}] \leq  \int_0^{1-p} f_{\varepsilon_m}(x) dx,
\end{align*}
and (\ref{jetzt2}) follows by letting $ m $ tend to infinity, by monotone convergence. We are thus left to show that
\begin{equation}
\label{variance}
\lim_{m \rightarrow \infty} \lim_{n \rightarrow \infty} \textup{Var}\left( \frac{1}{n} \sum_{i=1}^n Y_i(T(n))^ {\alpha} \cdot 1_{E_{n,m}} \right) =0.
\end{equation} We split $ \textup{Var}\left( \frac{1}{n} \sum_{i=1}^n Y_i(T(n))^{\alpha } \cdot 1_{E_{n,m}} \right)  $ in three parts
\begin{eqnarray*}
V_{n,m}^1 & = & \frac{1}{n^2} \sum_{i=1}^n \mathbb{E}[Y_i(T(n))^{2 \alpha} \cdot 1_{E_{n,m}}], \\
 V_{n,m}^2 & = & \frac{2}{n^2}  \sum_{i \neq j}^n \textup{Cov}( 1_{E_{n,m}}  Y_i(T(n))^\alpha , 1_{E_{n,m}}   Y_j(T(n))^\alpha  ), \\
 V_{n,m}^3 & = & \frac{1}{n^2}  \sum_{i=1}^n \mathbb{E}[Y_i(T(n))^{ \alpha  } \cdot 1_{E_{n.m}}]^2,
\end{eqnarray*}
and  show separately that all of them converge to zero as $ n $ tends to infinity. Since $ 2 \alpha < 4 $, we have
\begin{eqnarray*}
\lim_{m \rightarrow \infty} \lim_{n \rightarrow \infty} V_{n,m}^1 & \leq &\lim_{m \rightarrow \infty} \lim_{n \rightarrow \infty} \frac{1}{n^2} \sum_{i=1}^n \mathbb{E}[\overline{X}_i(n, \varepsilon_m )^{2 \alpha}  ] \\ & \leq & \lim_{m \rightarrow \infty} \lim_{n \rightarrow \infty}\frac{1}{n^2} \sum_{i=1}^n \mathbb{E}[\overline{X}_i(n, \varepsilon_m)^{4}  ]^{\frac{2 \alpha}{4}}.
\end{eqnarray*}
The fourth moment of a $ \textup{Geom}(q) $ random variable is given by $  1 + 24\frac{(1-q)^4}{q^4} + 60 \frac{(1-q)^3}{q^3}+ 50 \frac{(1-q)^2}{q^2}+15\frac{1-q}{q} $, which is smaller than $ 150 \frac{1}{q^4} $. Thus
\begin{eqnarray*}\lim_{m \rightarrow \infty} \lim_{n \rightarrow \infty} V_{n,m}^1 & \leq & \lim_{m \rightarrow \infty} \lim_{n \rightarrow \infty}\frac{1}{n^2}\sum_{i=1}^{n} \left(150\left(\frac{n(1-p)(1+\varepsilon_m)}{i}\right)^{4p}\right)^{\frac{2 \alpha}{4}} \\ & \leq &\lim_{m \rightarrow \infty} \lim_{n \rightarrow \infty} c_m \cdot \frac{1}{n^{2-2 \alpha p}} \sum_{i=1}^n \left( \frac{1}{i} \right)^{2 \alpha p}, \end{eqnarray*}
where $ c_m:=150^{\frac{\alpha}{2}}((1-p)(1+ \varepsilon_m))^{2 \alpha p} $. Recall that $ 2 > 2p\alpha $ since $ p < \frac{1}{\alpha} $. Now if $ 2 \alpha p>1 $ the series converges and if $ 2 \alpha p =1 $ the series is the harmonic series, whose partial sum grows logarithmically, and thus $ V_{n,m}^1 $ tends to zero as $ n $ and $ m $ tend to infinity. If on the other hand $ 2 \alpha p <1 $ , the series is the $ 2 \alpha p $-harmonic series, which can be bounded by:
\begin{eqnarray*}
\lim_{m \rightarrow \infty} \lim_{n \rightarrow \infty} V_{n,m}^1 &\leq &\lim_{m \rightarrow \infty} \lim_{n \rightarrow \infty} c_m \cdot  \frac{1}{n^{2-2 \alpha p}} \left( 1 + \frac{n^{1-2p \alpha} -1}{1-2p \alpha} \right)=0 ,
\end{eqnarray*}
see Chlebus \cite{harmonic}. Since  $ V^3_{n,m} \leq V^1_{n,m} $  we conclude that $ V^3_{n,m} $ tends to zero as $ n $ and $ m $ tend to infinity. Last we have
\begin{align*} 
\lim_{m \rightarrow \infty} \lim_{n \rightarrow \infty} V^2_{n,m}
\leq  \lim_{m \rightarrow \infty} \lim_{n \rightarrow \infty} \frac{2}{n^2}  \sum_{i \neq j} & \bigg( \mathbb{E}[\overline{X}_i(n,\varepsilon_m)^{\alpha}] \mathbb{E}[\overline{X}_j(n, \varepsilon_m)^{\alpha}] \\ &- \mathbb{E}[\underline{X}_i(n, \varepsilon_m)^{\alpha}] \mathbb{E}[\underline{X}_j(n,\varepsilon_m)^{\alpha}] \bigg)=0,
\end{align*}
and we have shown Equation (\ref{variance}). 
\end{proof}
\subsection{ Critical case $ \alpha p = 1 $}
In this section we aim to show the following result:
\begin{theorem} Let $ p<1 $, let $ \alpha p = 1 $ and let $ t \in \mathbb{R}_+ $. We then have the convergence in distribution \[ \lim_{n \rightarrow \infty}  \left(\frac{1}{n^t\cdot \log(n)}\right)^{\frac{1}{\alpha}} S_{\lfloor n^t \rfloor} =t^{\frac{1}{\alpha}} \cdot \tilde{ S},\] where $ \tilde{S} $ is a $ d $-dimensional isotropic $ \alpha $-stable distributed random variable with zero shift and scale parameter $ ((1-p)\Gamma(\alpha+1))^{\frac{1}{\alpha}} $, that is
\[\mathbb{E}\left[e^{i \langle \theta, \tilde{S} \rangle} \ \right]= \exp\left(-((1-p)\Gamma(\alpha+1)) \cdot \Vert \theta \Vert^\alpha\right), \quad \theta \in \mathbb{R}^d.\]
\end{theorem}
Recall that $ Y_i(t) $ denotes the population size of individuals of type $ i \in \mathbb{N} $ in our Yule process with mutation probability $ 1-p=1-\frac{1}{\alpha} $, and that for  $T(n) := \inf\lbrace t \in \mathbb{R} : Y(t) = n \rbrace $, we  have  

\[ \vert c_{i,n} \vert \overset{d}{=} Y_i(T(n))  \quad  \textnormal{for all} \quad  n \in \mathbb{N}. \]
By the same reasoning as in the subcritical case it will suffice to show the following statement.
\begin{lemma}
\label{critical one dimensional}
Let $ 0<\delta < 1 $, and let the sequence of events $ G_{n,m}:=E_m(\lfloor n^\delta \rfloor,n) $ be defined  as in Lemma \ref{sequuence}. We have the convergence in probability
\[\lim_{m \rightarrow \infty} \lim_{n \rightarrow \infty} \frac{1}{n\log(n)} \sum_{i=1}^n Y_i(T(n))^\alpha \cdot 1_{G_{n,m}} = (1-p) \Gamma(1+\alpha).
\]
\end{lemma}
\begin{proof}
If $ \alpha = 2 $ the results follows by a direct computation using the first and second moment of the geometric random variable in Lemma  \ref{sequuence}. Let us hence assume that $ \alpha < 2 $. As in the proof of Lemma \ref{expectations second moment} we shall  use the second moment method. First, we aim to show that
\begin{equation}
\label{ttt}
\lim_{m \rightarrow \infty}  \lim_{n \rightarrow \infty}  \frac{1}{n\log(n)} \sum_{i=1}^n \mathbb{E} \left[ Y_i(T(n))^\alpha \cdot 1_{G_{n,m}} \right] = (1-p) \Gamma(1+\alpha).
\end{equation}
For the upper bound  note that   \[ \mathbb{E} \left[ Y_i(T(n))^\alpha \cdot 1_{G_{n,m}} \right]  \leq \mathbb{E}\left[ \vert c_{i,n} \vert^\alpha \right]
,  \] and using the same bounds as in  inequality (\ref{jetzt}), and keeping in mind that $ \alpha p=1 $, we have 
\begin{eqnarray}
\label{W}
\frac{1}{n\log(n)} \sum_{i=1}^{\lfloor n^\delta \rfloor} \mathbb{E}\left[ \vert c_{i,n} \vert^\alpha \right] & \leq &  (3c_0+i_0^\alpha) \cdot \frac{1}{\log(n)} \sum_{i=1}^{\lfloor n^\delta \rfloor} \frac{1}{i},
\end{eqnarray}
where $ i_0 $ and $ c_0 $ are constants, and thus
\begin{equation} 
\label{1} 
\limsup_{n \rightarrow \infty}  \frac{1}{n\log(n)} \sum_{i=1}^{\lfloor n^\delta \rfloor} \mathbb{E}\left[ \vert c_{i,n} \vert^\alpha  \right] \leq (3c_0+i_0^\alpha)\delta. 
\end{equation}
Now recall  that $ \vert c_{k,n}^ {\prime} \vert  $ denotes the size of the subtree rooted at node $ k $, and that it is a cluster if the edge connecting $ k $ to its parent is deleted.  We have seen in Equation (\ref{subtree rooted}) that   \[ \vert c_{k,n}^\prime \vert  \overset{d}{=} \vert c_{1,Y(n,k)} \vert ,\] where $Y(n,k)  $ is Beta-Binomial distributed  with parameter $ (n-k,1,k-1) $ and further independent of $ \vert c_{1,i} \vert $ for all $ i \in \mathbb{N} $. Thus
\begin{eqnarray*}
\sum_{i=\lfloor n^\delta \rfloor +1}^{n} \mathbb{E}\left[ \vert c_{i,n} \vert^\alpha  \right] & = &  (1-p) \sum_{i=\lfloor n^\delta \rfloor +1}^{n} \mathbb{E}\left[ \vert c^\prime_{i,n} \vert^\alpha  \right] \\ & = & (1-p)\sum_{i=\lfloor n^\delta \rfloor+1}^{n} \sum_{k=1}^{n-i} \mathbb{E}\left[ \vert c_{1,k} \vert^\alpha  \right] \mathbb{P}(Y(n,i)=k).
\end{eqnarray*}
Let $ \varepsilon > 0 $, by Equation (\ref{limit expectation2}), there exists an integer number $ k_0 $ such that for all $ k \geq k_0 $:
\[\mathbb{E}\left[ \vert c_{1,k} \vert^\alpha  \right] \leq k \cdot \mathbb{E}\left[X_1^\alpha\right](1+\varepsilon) = k \cdot \Gamma(\alpha + 1)(1+\varepsilon), \]
and we thus have the inequality
\begin{eqnarray*}
 \sum_{k=k_0}^{n-i} \mathbb{E}\left[ \vert c_{1,k} \vert^\alpha  \right] \mathbb{P}(Y(n,i)=k) & \leq & \Gamma(\alpha + 1)(1+\varepsilon)  \sum_{k=k_0}^{n-i} k \cdot  \mathbb{P}(Y(n,i)=k) \\  & \leq &\Gamma(\alpha + 1) (1+\varepsilon) \cdot \mathbb{E}\left[Y(n,i)\right] \\ & \leq &\Gamma(\alpha + 1) (1+\varepsilon)  \cdot n \cdot \frac{1}{i}.
\end{eqnarray*}
Using the elementary bound
\begin{eqnarray*}
\sum_{i=\lfloor n^\delta \rfloor +1}^{n} \sum_{k=1}^{k_0-1} \mathbb{E}\left[ \vert c_{1,k} \vert^\alpha  \right] \mathbb{P}(Y(n,i)=k) \leq  n  k_0 \cdot k_0^\alpha,
\end{eqnarray*}
we thus have 
\begin{eqnarray*}
\sum_{i=\lfloor n^\delta \rfloor +1}^{n} \mathbb{E}\left[ \vert c_{i,n} \vert^\alpha  \right] \leq (1-p) \left( n  k_0^{\alpha +1} + n  \Gamma(\alpha + 1)(1+\varepsilon) \sum_{i=\lfloor n^\delta \rfloor +1}^{n} \frac{1}{i} \right).
\end{eqnarray*}
Moreover
\begin{equation} \label{K} \limsup_{n \rightarrow \infty}  \frac{1}{n \log(n)} \sum_{i=\lfloor n^\delta \rfloor +1}^{n} \mathbb{E}\left[ \vert c_{i,n} \vert^\alpha  \right]   \leq \Gamma(\alpha + 1) (1+\varepsilon) (1-\delta),
\end{equation}
and combining Inequality (\ref{1}) and Inequality (\ref{K}) and letting $ \delta $ and $ \varepsilon $ tend to zero we conclude that
\[\limsup_{n \rightarrow \infty}  \frac{1}{n \log(n)} \sum_{i=1}^{n}  \mathbb{E}\left[ \vert c_{i,n} \vert^\alpha  \right]   \leq \Gamma(\alpha + 1) (1-p) .\]
For the lower bound we first aim to show that
\begin{equation} 
\label{probab}
\lim_{m \rightarrow \infty} \limsup_{n\rightarrow \infty} \frac{1}{n \log(n)}\sum_{i=1}^n \mathbb{E}\left[ Y_i(T(n))^\alpha 1_{\lbrace \Omega \backslash G{n,m}\rbrace}\right]=0. 
\end{equation}
Indeed, by the Cauchy-Schwarz inequality we have
\begin{align*}
\mathbb{E} & \left[ Y_i(T(n))^\alpha 1_{\lbrace \Omega \backslash G{n,m}\rbrace }\right] \\ &\leq  \left( \mathbb{E}\left[ \vert c_{i,n} \vert^{2 }\right] \right)^{\alpha /2} \cdot  \left(\mathbb{P}(\Omega \backslash G_{n,m})\right)^{\alpha / 2}.  
\end{align*}
By the same reasoning as before there exists an integer number $ k_0 $ such that
\begin{align*}
\frac{1}{n}\mathbb{E}\left[ \vert c_{i,n} \vert^{2 }\right] \leq   k_0^{3}+ \Gamma(3) (1+\varepsilon)  \mathbb{E}[Y(n,i)^2], 
\end{align*}
and since 
\[  \mathbb{E}[Y(n,i)^2] =\frac{(n-i)(2(n-i)+i-1))}{i(i+1)}  \leq  \frac{3n^2}{i^2}\]
Equation \ref{probab} follows since $ \mathbb{P}(\Omega \backslash G_{n,m}) \rightarrow 0 $ as $ n $ and $ m $ tend to infinity. The lower bound can now be shown in the same spirit as the upper bound. We are left to show that
\begin{equation}
\label{Varri}
\lim_{m \rightarrow \infty} \lim_{n \rightarrow \infty}\textup{Var}\left(  \frac{1}{n\log(n)} \sum_{i=1}^n  Y_i(T(n))^\alpha \cdot 1_{G_{n,m}} \right)=0.
\end{equation}
We split the variance in three parts
\begin{eqnarray*}
V_{n,m}^1 & = & \frac{1}{(n \log(n))^2} \sum_{i=1}^n \mathbb{E}[ Y_i(T(n))^{2 \alpha} \cdot 1_{G_{n,m}}], \\
 V_{n,m}^2 & = & \frac{2}{(n\log(n))^2}  \sum_{i \neq j}^n \textup{Cov}( 1_{G_{n,m}} Y_i(T(n))^\alpha , 1_{G_{n,m}} Y_j(T(n))^\alpha  ), \\
 V_{n,m}^3 & = & \frac{1}{(n\log(n))^2}  \sum_{i=1}^n \mathbb{E}[Y_i(T(n))^{ \alpha  } \cdot 1_{G_{n,m}}]^2.
\end{eqnarray*}
By the same computations as in the proof of Lemma \ref{expectations second moment} we have  
\begin{eqnarray*}
\lim_{m \rightarrow \infty} \lim_{n \rightarrow \infty}  V_{n,m}^1 \leq \lim_{m \rightarrow \infty} \lim_{n \rightarrow \infty} c_m\cdot \frac{1}{\log(n)^ 2}\sum_{i=1}^ n \frac{1}{i^ 2}=0,
\end{eqnarray*}
where $ c_m=150^{\frac{\alpha}{2}}((1-p)(1+ \varepsilon_m))^{2} $ and hence $ V_{n,m}^3 $ also tends to zero as $ n $ and $ m $ tend to infinity. Last
\begin{align*}
\lim_{m \rightarrow \infty} \lim_{n \rightarrow \infty}  V_{n,m}^2 \leq \lim_{m \rightarrow \infty} \lim_{n \rightarrow \infty} \bigg(  & \frac{2}{(n\log(n))^2}  \sum_{i \neq j} \bigg( \mathbb{E}[\overline{X}_i(n,\varepsilon_m)^{\alpha}] \mathbb{E}[\overline{X}_j(n, \varepsilon_m)^{\alpha}]  \\  &- \mathbb{E}[\underline{X}_i(n, \varepsilon_m)^{\alpha}]\mathbb{E}[\underline{X}_j(n,\varepsilon_m)^{\alpha}] \bigg)\bigg) =0,
\end{align*}
and we find that Equation (\ref{Varri}) holds.
\end{proof}
\section{A characterization of the limiting process of the Shark Random Swim}

In this section we characterize the limiting process of the Shark Random Swim in dimension one, in the critical and subcritical case. Recall that for the Elephant Random Walk the limiting process in the critical and subcritical case is a Gaussian process. We shall see that for  the Shark Random Swim this role is played by $ \alpha $-stable processes, and we thus first give some background on stable processes.
\subsection{Some results on stable processes}
\label{sttabel}
We first give a definition of stable processes.
\begin{definition}
We call the process $ (S(t), t \geq 0) $  a (strictly) stable process, if all its finite-dimensional distributions, that is the distribution of the  vectors 
\[(S(t_1),...,S(t_k)), \quad k \in \mathbb{N}, \quad t_1,...,t_k \in \mathbb{R}_+,\]
are (strictly) stable. 
\end{definition}
It is well-known that for strictly stable processes, this definition is equivalent to the following characterization. For a proof see for example Samorodnitsky and Taqqu \cite{Stable}, Theorem 3.2.1.

\begin{lemma}
The stochastic process $ (S(t), t \geq 0) $ is a strictly stable process, if and only if all linear combinations, 
\[\sum_{i=1}^k a_k S(t_k), \quad t_1,...,t_k \in \mathbb{R}_+, \quad a_1,...,a_k \in \mathbb{R}, \quad k \in \mathbb{N},\]
are strictly stable.

\end{lemma}
We will  be interested in stable integrals, a special class of stable processes, which will enable us to characterize the limiting process of the Shark Random Swim in the subcritical case. Following Samorodnitsky and Taqqu \cite{Stable} Chapter 3, we let   $ (E,\mathcal{E},m) $ denote a measure space and define the linear space
\begin{align*} L^{\alpha}(E,\mathcal{E},m):=\left\lbrace \textnormal{$ h: h $ measurable and $ \int_E \vert h(x) \vert^\alpha m(dx) < \infty $ }\right\rbrace.
\end{align*} By Kolmogorov's existence theorem one then has the following statement.

\begin{lemma}

 There exists a stochastic process \[ (I(h), h \in L^\alpha(E,\mathcal{E},m)) ,\]  whose finite-dimensional distributions $ (I(h_1),...I(h_k)) $, with $ k \in \mathbb{N}  $, have characteristic function
\begin{equation} 
\label{characteristic}
 \exp\left(-\int_E \left\vert \sum_{i=1}^k \theta_i h_i(x) \right\vert^\alpha  m(dx) \right).
 \end{equation} The random variable $ I(h) $ is called the $ \alpha $-stable integral of $ h $ with control measure $ m $ and skewness intensity zero. 
\end{lemma} 
As the name suggests, $ I(h) $ for $ h \in L^\alpha(E,\mathcal{E},m) $, can also be viewed as an integral, and we now aim to define the measure which  plays the role as an integrator, the so called $ \alpha $-stable random measure. Let $ (\Lambda,\mathbb{Q},\mathcal{G}) $ denote the underlying probability space,  $ L^0(\Lambda) $ the set of all real valued random variables defined on it, and define \[ \mathcal{E}_0:=\lbrace A \in \mathcal{E}: m(A) < \infty \rbrace   .\]
The $ \alpha $-stable random measure is the defined as follows.
\begin{definition}
We call the  set function 
\[M: \mathcal{E}_0 \rightarrow L^0(\Lambda)\]
an $ \alpha $-stable random measure with control measure $ m $ and skewness intensity zero, if

\begin{enumerate}

\item it is sigma-additive and independently scattered (i.e. for $ A_1,...,A_n  $ disjoint, $ M(A_1),...M(A_n) $ are independent random variables), 
\item for each $ A \in \mathcal{E}_0 $ the random variable  $ M(A) $ is $ \alpha  $-stable distributed with scale parameter $ m(A)^{\frac{1}{\alpha}} $, skewness parameter zero, and shift parameter zero.
\end{enumerate}
\end{definition}

 There  is  an equivalent definition of the  $ \alpha $-stable integral as the limit in probability  \[ \lim_{n \rightarrow \infty}\int_Eh^n(x)dM(x), \] where $ h^n $ is a sequence of simple functions approximating $ h $, and $ M $ an $ \alpha $-stable random measure with control measure $ m $ and skewness intensity zero, (see  Samorodnitsky and Taqqu \cite{Stable}, Chapter 3.4). To stress the connection with the stable random measure, one also uses the notation

  \[I(h)=\int_Eh(x)dM(x).\] 
For the distribution of $ I(h) $, it has been shown that:
\begin{lemma}

The random variable $ I(h) $, for $ h \in L^{\alpha}(E,\mathcal{E},m) $, is $ \alpha $-stable distributed with scale parameter \[\left(\int_E \vert h(x)\vert^\alpha m(dx) \right)^{\frac{1}{\alpha}},\]
skewness parameter zero, and shift parameter zero. Moreover for $ h_1,...,h_k \in L^{\alpha}(E,\mathcal{E},m) $, the integrals $ I(h_1),..,I(h_k) $ are jointly $ \alpha $-stable distributed with characteristic function given by Equation (\ref{characteristic}).
\end{lemma} 
A proof of this lemma can be found in Samorodnitsky and Taqqu \cite{Stable} Property 3.2.1. In the Gaussian case, $ \alpha =2 $, we will further need to compute the covariance functions of the process, which are given in the following lemma. A proof can be found in Samorodnitsky and Taqqu \cite{Stable}, Proposition 3.5.2.
\begin{lemma}
\label{covariance}
Let $ \alpha = 2 $ and  $ h_1,h_2 \in L^{2}(E,\mathcal{E},m) $, then
\[\textup{Cov}(I(h_1),I(h_2))=2\int_E h_1(x)h_2(x)m(dx).\]
\end{lemma}
Recall that a stable random variable with scale parameter $ c \in (0,\infty)$, skewness parameter zero, and shift parameter zero, is  normal distributed $\mathcal{N}(0, 2c^2) $, which explains the factor $ 2 $.
\subsection{Subcritical case $ \alpha p < 1 $}
In this section we shall study the subcritical case. Recall that the position of the Shark Random Swim at time $ n $, in dimension one, can  be expressed as \[ S_n=\sum_{i} \vert c_{i,n} \vert \xi_i ,\] where   $ \xi_i$, for $ i \in \mathbb{N} $, are i.i.d. one-dimensional  standard isotropic strictly stable random variables with zero shift and stability parameter $ \alpha \in (0,2] $, that is \[\mathbb{E}\left[e^{i  \langle \theta , \xi_1 \rangle} \right]=e^{- \vert \theta \vert^\alpha}\quad \textnormal{for} \quad \theta \in \mathbb{R}.\]   In this section we aim to show that :

\begin{theorem}
\label{theorem 2}
Let $ \alpha p < 1 $ and $ 0 < t_1 \leq t_2 \leq...\leq t_k \in \mathbb{R}_+ $ for $ k \in \mathbb{N} $. We have the distributional convergence 
\[\left( \left( \frac{1}{n}\right)^{\frac{1}{\alpha}}  S_{\lfloor t_1 n \rfloor },...,\left( \frac{1}{n}\right)^{\frac{1}{\alpha}}  S_{\lfloor t_k n \rfloor }\right) \rightarrow (S(t_1),...,S(t_k)), \]
where the vector  $( S(t_1),...,S(t_k)) $ is jointly strictly $ \alpha $-stable distributed.

\end{theorem}
 For the proof of Theorem \ref{theorem 2}, we aim to show that each linear combination  \[ \sum_{j=1}^k a_j  S(t_j), \quad  a_1,..,a_k \in \mathbb{R}\] is strictly $ \alpha $-stable distributed. Recall that this means, that is there exists a constant  $  g(a_1,...,a_k,t_1,...,t_k) $  such that for each $ \theta \in \mathbb{R} $ we have
\begin{equation}
\label{m}
\mathbb{E}\left[\exp \left(i \theta \sum_{j=1}^k a_j   S(t_j) \right) \right]=\exp\left(- \vert \theta  \vert^{\alpha} \cdot g(a_1,...,a_k,t_1,...,t_k) \right)
.\end{equation}

In order to prove Equation (\ref{m}), recall that we defined the filtration $ \mathcal{F}_n: = \sigma(\vert c_{1,n} \vert,...,\vert c_{n,n} \vert) $.  By conditioning we then have
\begin{align}
\label{aa}
\mathbb{E}&\left[\exp\left( i \theta \sum_{j=1}^k  a_j \left( \frac{1}{n}\right)^{\frac{1}{\alpha}}  S_{\lfloor t_j n \rfloor }\right) \bigg\vert \mathcal{F}_{\lfloor t_kn \rfloor} \right] \nonumber \\&= \exp\left(- \vert \theta \vert^\alpha \cdot \frac{1}{n}\sum_{m=1}^\infty \bigg\vert \sum_{j=1}^k a_j \cdot \vert c_{m,\lfloor t_j n \rfloor } \vert 1_{\lbrace m \leq \lfloor t_j n \rfloor  \rbrace } \bigg\vert^\alpha \right),
\end{align}
and it will be enough to control the limiting behavior of (\ref{aa}). Recall   that in Equation (\ref{jetzt}), we established a bound for the cluster sizes, which we will state now as a lemma, as we will use it repeatedly.
\begin{lemma} 
\label{boundibound}
For each $ \alpha \leq 2  $ there exists a constant $ c_1(\alpha,p) $ such that for all $ k \geq 2 $ we have
 \begin{equation*}  \frac{1}{n^{\alpha p}}\mathbb{E}\left[ \vert c_{k,n}^{\prime} \vert^{\alpha} \right] \leq c_1(\alpha,p) \cdot \frac{1}{k^{\alpha p}}.
\end{equation*} 
\end{lemma}
As in Section \ref{erdb}, we denote by $ Y_i(t) $ the population size of individuals of type $ i \in \mathbb{N} $ in a Yule process with mutation rate $ 1-p $. Recall moreover that  \[ b_i:=\inf\lbrace t \geq 0: Y_i(t) >0 \rbrace  \] denotes the birth time of the first individual of type $ i $, and    \[T(n) := \inf\lbrace t \in \mathbb{R} : Y(t) = n \rbrace\] denotes the first time when there are $ n $ individuals alive. The connection with the cluster sizes is then given by the equality in distribution
\begin{equation}
\label{equation1}
(Y_1(T(n)),...,Y_n(T(n))_{n\geq 0}   \overset{d}{=}  (\vert c_{1,n} \vert,...,\vert c_{n,n} \vert)_{n \geq 0} .  \end{equation}

Now let us  define the map $ f:(0, \infty) \times \Omega \rightarrow \mathbb{R}_+ $ with
\begin{align}
\label{functionn}
f(x,\omega):= Y_1(-\log(x)+\log(1-p))(\omega)1_{\lbrace x < (1-p) \rbrace },
\end{align}
and note that $ f $ is measurable since for $ \omega $ fixed, the map $ x \mapsto  f(x,\omega)$ is right-continuous and for fixed $ x $ the map $ \omega \mapsto  f(x,\omega)$ is measurable on $ \Omega $. Moreover
\[\int_{(0,\infty) \times \Omega} \vert f(x,\omega) \vert^\alpha dx \times d\mathbb{P}(\omega) \leq \int_{(0,1)} \left( \frac{1}{y}\right)^{\alpha p} dy < \infty,  \]
by Tonelli's theorem and Lemma \ref{Yule}, since $ \alpha p < 1 $. 

Now 
by the Equality in distribution  (\ref{equation1}), it will be enough to prove the following result.

\begin{proposition}
\label{Prop1}
Let $ k \in \mathbb{N} $,  $ 0 < t_1 \leq t_2 \leq...\leq t_k \in \mathbb{R}_+ $,  $ a_1, a_2,...,a_k \in \mathbb{R} $. There exists a   sequence of events $ E_{n,m} $, with \[ \lim_{m \rightarrow \infty} \lim_{n \rightarrow \infty} \mathbb{P}(E_{n,m})= 1, \] such that  we have the convergence in probability
\begin{align*}
\lim_{m \rightarrow \infty} \lim_{n \rightarrow \infty} \frac{1}{n} \sum_{i=1}^\infty \bigg\vert \sum_{j=1}^k & a_j \vert Y_i(T(\lfloor t_j n \rfloor) \vert 1_{ \lbrace  i \leq  \lfloor t_j n \rfloor\rbrace } \bigg\vert^\alpha 1_{E_{n,m}}  \\ &= \int_{(0,\infty) \times \Omega} \bigg\vert \sum_{j=1}^k a_jf\left(\frac{x}{t_j}, \omega \right)\bigg\vert^\alpha dx \times d\mathbb{P}(\omega).
\end{align*}

\end{proposition}
\begin{proof}
We first construct the sequence of events. Let  \[ E_{m}(\lfloor \delta n t_k \rfloor, \lfloor n t_k \rfloor ) , \quad \textnormal{with} \quad 0 < \delta < \min\left(1, \frac{1}{t_k} \right) \] be defined as in Lemma \ref{sequuence}, and let $ A_{n,m} $ denote the event on which for all $ i \in \lbrace \lfloor \delta n t_k \rfloor,..., \lfloor n t_k \rfloor\rbrace $ the random variables
\begin{align} \label{same sign} \frac{1}{n^p} \sum_{j=1}^k a_j Y_{i}(T(\lfloor t_j n \rfloor)  1_{ \lbrace  i \leq  \lfloor t_j n \rfloor\rbrace } \quad \textnormal{and} \quad \frac{1}{n^p} \sum_{j=1}^k a_j \overline{X}_{i}(\lfloor t_j n \rfloor, \varepsilon_m)  1_{ \lbrace  i \leq  \lfloor t_j n \rfloor\rbrace } 
\end{align}
have the same sign, and note that, by definition, these two random variables have a.s. the same limit as $ m $ and $ n $ tend to infinity.
We then define
\[E_{n,m}:=  E_{m}(\lfloor \delta n t_k \rfloor, \lfloor n t_k \rfloor ) \cap A_{n,m}. \]
For the rest of the proof we  shorthand write:
\begin{equation*}
f_j(x,\omega):= f\left(\frac{x}{t_j}, \omega \right).
\end{equation*}
We shall use the second moment method.  First we show that: 
\begin{align}
\lim_{m \rightarrow \infty} \lim_{n \rightarrow \infty} & \frac{1}{n} \sum_{i=1}^\infty \mathbb{E} \left[ \bigg\vert \sum_{j=1}^k a_j Y_{i}(T(\lfloor t_j n \rfloor) ) 1_{ \lbrace  i \leq  \lfloor t_j n \rfloor\rbrace } \bigg\vert^\alpha 1_{E_{n,m}} \right] \nonumber  \\ &=\int_{(0,\infty) \times \Omega} \bigg\vert \sum_{j=1}^k a_j f_j(x,\omega)\bigg\vert^\alpha dx \times d\mathbb{P}(\omega).
\end{align}
We split the sum into two parts. For the first part note that
\begin{align*}
\sum_{i=1}^{\lfloor \delta t_k n \rfloor} & \mathbb{E} \left[ \bigg\vert \sum_{j=1}^k a_j Y_{i}(T(\lfloor t_j n \rfloor) ) 1_{ \lbrace  i \leq  \lfloor t_j n \rfloor\rbrace } \bigg\vert^\alpha 1_{E_{n,m}} \right]  \\ &\leq \sum_{i=1}^{ \lfloor \delta t_k n  \rfloor }\mathbb{E} \left[ \bigg\vert \sum_{j=1}^k a_j \vert c_{i,\lfloor t_j n \rfloor  } \vert 1_{\lbrace i \leq \lfloor t_j n \rfloor \rbrace}  \bigg\vert^\alpha \right]. 
\end{align*}
Now, letting   $ a:=(\vert a_1 \vert+...+\vert a_k \vert ) $, we have 
\begin{align*}
\frac{1}{n} \sum_{i=1}^{ \lfloor \delta t_k n  \rfloor }\mathbb{E} \left[ \bigg\vert \sum_{j=1}^k a_j \vert c_{i,\lfloor t_j n \rfloor  } \vert 1_{ \lbrace i \leq \lfloor t_j n \rfloor \rbrace }  \bigg\vert^\alpha \right] &\leq \frac{1}{n} \sum_{i=1}^{ \lfloor \delta t_k n \rfloor }\mathbb{E} \left[ \bigg\vert k \cdot a \cdot \vert c_{i,\lfloor t_k n \rfloor  } \vert \bigg\vert^\alpha \right] \\ &\leq (k \cdot a)^\alpha \frac{1}{n} \sum_{i=1}^{ \lfloor \delta t_k n \rfloor }\mathbb{E} \left[  \vert c_{i,\lfloor t_k n \rfloor  } \vert^\alpha \right].
\end{align*}
By Lemma \ref{boundibound}   there exists a constant $ c_1(\alpha,p) $ such that  we have the bound:
\begin{align*}
(k \cdot a)^\alpha \frac{1}{n} \sum_{i=1}^{ \lfloor \delta t_k n \rfloor }\mathbb{E} \left[  \vert c_{i,\lfloor t_k n \rfloor  } \vert^\alpha \right] &\leq (k \cdot a)^\alpha  \cdot \frac{t_k^{\alpha p} c_1(\alpha,p)}{n^{1-\alpha p}} \sum_{i=1}^{\lfloor t_k \delta n \rfloor} \frac{1}{i^{\alpha p}} \\ &\leq  (k \cdot a)^\alpha  \cdot \frac{t_k^{\alpha p} c_1(\alpha,p)}{n^{1-\alpha p}}\left(1+\frac{(t_kn\delta)^{1-\alpha p}-1}{1-\alpha  p}\right),
\end{align*}
where the last inequality is due to Chlebus \cite{harmonic}. We conclude that
\begin{equation}
\label{dellta}
\limsup_{n \rightarrow \infty} \frac{1}{n} \sum_{i=1}^{ \lfloor \delta t_k n \rfloor }\mathbb{E} \left[ \bigg\vert \sum_{j=1}^k a_j \vert c_{i,\lfloor t_j n \rfloor  } \vert \bigg\vert^\alpha \right]  \leq  (k \cdot a)^\alpha  \cdot t_k^{\alpha p} c_1(\alpha,p) \cdot  \frac{(\delta   t_k)^{1- \alpha p} }{1-\alpha p}. 
\end{equation}
For the other part of the sum,  denote by $ \sgn $ the signum function, and   define  for $ j \in \lbrace 1,...,k \rbrace $  and $ i \in \lbrace  \lfloor \delta t_k n \rfloor,..., \lfloor t_k n \rfloor \rbrace $ the event
\[A_{i,j}:=\left\lbrace \omega \in \Omega: \sgn(a_j)=\sgn\left(\sum_{m=1}^k a_m \overline{X}_{i}(\lfloor t_m n \rfloor, \varepsilon_m) (\omega)) 1_{ \lbrace  i \leq  \lfloor t_m n \rfloor\rbrace }  \right) \right\rbrace,\]
and the random variables
\begin{align}
\label{chi}
\overline{\chi}_i(\lfloor t_j n \rfloor, \varepsilon) &:=\overline{X}_i( \lfloor t_j n \rfloor,\varepsilon)1_{A_{i,j}}+\underline{X}_i( \lfloor t_j n \rfloor ,\varepsilon)1_{A_{i,j}^c} \\ \nonumber \underline{\chi}_i( \lfloor t_j n \rfloor ,\varepsilon)&:=\underline{X}_i(  \lfloor t_j n \rfloor,\varepsilon)1_{A_{i,j}}+\overline{X}_i(  \lfloor t_j n \rfloor,\varepsilon)1_{A_{i,j}^c}.
\end{align}
Let $ c_1,...,c_k, d_1,...,d_k \in \mathbb{R} $. Note that if $ \sum_{j=1}^k a_j \cdot c_j > 0 $ then
\[ \sum_{j=1}^k a_j \cdot c_j \leq \sum_{j=1}^k a_j \cdot d_j \]  if for $ j $ s.t. $ a_j >0 $ one has $ d_j > c_j  $ and for $ j $ s.t. $ a_j < 0 $  one has $ d_j < c_j $. The opposite is true if $ \sum_{j=1}^k a_j \cdot c_j < 0 $. Hence on the event $ E_{n,m} $ we have, thanks to Property (\ref{same sign}), the  inequalities
\begin{align} 
\label{bounds}
 \bigg\vert \sum_{j=1}^k a_j \cdot \underline{\chi}_i ( \lfloor t_j n \rfloor ,\varepsilon_m)  1_{ \lbrace  i \leq  \lfloor t_j n \rfloor\rbrace } \bigg\vert^\alpha  &\leq \bigg\vert \sum_{j=1}^k a_j Y_{i}(T(\lfloor t_j n \rfloor) 1_{ \lbrace  i \leq  \lfloor t_j n \rfloor\rbrace } \bigg\vert^\alpha \nonumber \\ &\leq   \bigg\vert \sum_{j=1}^k a_j  \cdot \overline{\chi}_i (\lfloor t_j n \rfloor, \varepsilon_m )  1_{ \lbrace  i \leq  \lfloor t_j n \rfloor\rbrace } \bigg\vert^\alpha.
\end{align}
Now define the functions $ f_j^{\varepsilon_m} : (0,\infty) \times \Omega \rightarrow \mathbb{R}_+$ with
\begin{align*}
f_j^{\varepsilon_m}(x,\omega):=f_j\left(\frac{x}{1+\varepsilon_m},\omega\right)1_{A_{j}(x, \varepsilon_m)}+f_j\left(\frac{x}{1-\varepsilon_m},\omega\right)1_{A_{j}(x, \varepsilon_m)^c},
\end{align*}
where $ A_{j}(x, \varepsilon_m) $, for $ x \in (0,(1-p)(1+\varepsilon_m)) $, is defined as the event on which 
\begin{equation*} \sgn(a_j)=\sgn \left(\sum_{m=1}^k a_m Y_1(-\log(x)+\log(1-p)+\log(1+\varepsilon_m) )  (\omega) \right).
\end{equation*}
Note that for all $ j \in \lbrace 1,...,k \rbrace$, the function $ f_j^{\varepsilon_m} $ is measurable, since the function $ f_j $ is measurable and $ A_{j}(x, \varepsilon_m) $ is a measurable set.  Recall that if $ i-1 < n(1-p)(1+\varepsilon_m) $ then
\begin{eqnarray}
\label{snuggles}
\overline{X}_i(n,\varepsilon) &=& Y_i(\log(n)-\log(i-1)+\log(1-p)+\log(1+\varepsilon)+b_i) ,
\end{eqnarray}
 and $ \overline{X}_i(n,\varepsilon)=0 $ otherwise. Recall moreover that
\begin{equation}
\label{wurmli}
 (Y_i(t +b_i ), t \geq 0) \overset{d}{=} (Y_1(t), t \geq 0), \quad \textnormal{ for all } \quad i \in \mathbb{N}. 
\end{equation} 
 Combining the Equations (\ref{bounds}), (\ref{snuggles}), and (\ref{wurmli}), we thus have
\begin{align*}
 \sum_{i=\lfloor \delta t_k n \rfloor }^\infty & \mathbb{E} \left[ \bigg\vert \sum_{j=1}^k a_j Y_{i}(T(\lfloor t_j n \rfloor) \bigg\vert^\alpha  1_{ \lbrace  i \leq  \lfloor t_j n \rfloor\rbrace } 1_{E_{n,m} } \right] \\ &\leq \sum_{i=\lfloor \delta t_k n \rfloor} ^{ \lfloor t_k n  \rfloor }\int_{\Omega} \bigg\vert \sum_{j=1}^k a_j f_j^{\varepsilon_m}\left(\frac{i-1}{n},\omega\right) \bigg\vert^\alpha  d\mathbb{P}. 
\end{align*}

Now since the $ f_j $ are measurable functions, we have that
\begin{align*}
\lim_{n \rightarrow \infty} \frac{1}{n}\sum_{i=\lfloor \delta t_k n \rfloor }^{ \lfloor t_k n  \rfloor }  & \int_{\Omega} \bigg\vert \sum_{j=1}^k a_j f^{\varepsilon_m}_j\left(\frac{i-1}{n},\omega\right) \bigg\vert^\alpha d\mathbb{P}(\omega) \\   &= \int_{(\delta t_k,t_k)}  \left( \int_{\Omega} \bigg\vert \sum_{j=1}^k a_j f^{\varepsilon_m}_j\left(x,\omega\right) \bigg\vert^\alpha d\mathbb{P}(\omega) \right) dx \\   &= \int_{(\delta t_k,\infty)}  \left( \int_{\Omega} \bigg\vert \sum_{j=1}^k a_j f^{\varepsilon_m}_j\left(x,\omega\right) \bigg\vert^\alpha d\mathbb{P}(\omega) \right) dx \\  &=   \int_{(\delta t_k,\infty) \times \Omega} \bigg\vert \sum_{j=1}^k a_j f^{\varepsilon_m}_j\left(x,\omega\right) \bigg\vert^\alpha dx \times d\mathbb{P}(\omega) , 
\end{align*}
where the last equality is due to Tonelli's theorem for nonnegativ, measurable functions. Adding the two parts of the sum, and then letting  $ \delta $ tend to zero and $ m $ tend to infinity we derive the upper bound. The lower bound can be shown in the same spirit. We are then left to show
\begin{equation}
\label{Variance}
\lim_{n \rightarrow \infty} \frac{1}{n^2} \cdot \textup{Var}\left(  \sum_{i= \lfloor \delta t_k n \rfloor }^\infty \bigg\vert \sum_{j=1}^k a_j Y_{i}(T(\lfloor t_j n \rfloor) 1_{ \lbrace  i \leq  \lfloor t_j n \rfloor\rbrace } \bigg\vert^\alpha 1_{E_{n,m}}\right)=0.
\end{equation}
We split the variance in three parts. Let
\begin{align*}
V_{n,m}^1 :&=\frac{1}{n^2} \cdot \sum_{i= \lfloor \delta t_k n \rfloor }^\infty  \mathbb{E}\left[\bigg\vert \sum_{j=1}^k a_j Y_{i}(T(\lfloor t_j n \rfloor) 1_{ \lbrace  i \leq  \lfloor t_j n \rfloor\rbrace } \bigg\vert^{2 \alpha} 1_{E_{n,m}} \right], 
\\ V_{n,m}^2 :&= \frac{2}{n^2} \cdot \sum_{i \neq l} \textup{Cov}\bigg( \bigg\vert \sum_{j=1}^k a_jY_{i}(T(\lfloor t_j n \rfloor) \vert 1_{ \lbrace  i \leq  \lfloor t_j n \rfloor\rbrace } \bigg\vert^{ \alpha} 1_{E_{n,m}}, \\ &\textnormal{\textcolor{white}{............................}} \bigg\vert \sum_{j=1}^k a_j Y_{l}(T(\lfloor t_j n \rfloor)  1_{ \lbrace  l \leq  \lfloor t_j n \rfloor\rbrace } \bigg\vert^{ \alpha} 1_{E_{n,m}}\bigg),
\\ V_{n,m}^ 3 :&= \frac{1}{n^2} \cdot \sum_{i= \lfloor \delta t_k n \rfloor }^\infty  \mathbb{E}\left[\bigg\vert \sum_{j=1}^k a_j Y_{i}(T(\lfloor t_j n \rfloor) 1_{ \lbrace  i \leq  \lfloor t_j n \rfloor\rbrace } \bigg\vert^{ \alpha} 1_{E_{n,m}} \right]^2.
\end{align*}
Recall that   $ a:=\vert a_1 \vert+...+\vert a_k \vert  $. We then have
\begin{align*}
\lim_{m \rightarrow \infty} \lim_{n \rightarrow \infty} V_{n,m}^1 &\leq \lim_{m \rightarrow \infty} \lim_{n \rightarrow \infty}  \frac{1}{n^2} \cdot  \sum_{i= \lfloor \delta t_k n \rfloor }^{ \lfloor t_k n \rfloor} \mathbb{E}\left[ \bigg\vert k  \cdot a \cdot  \overline{X}_i(\lfloor t_k n \rfloor, \varepsilon_m) \vert \bigg\vert^{2 \alpha} \right] \\ &\leq  \lim_{m \rightarrow \infty} \lim_{n \rightarrow \infty}  \frac{( k \cdot a)^{2\alpha} }{n^2} \cdot \sum_{i=1}^{ \lfloor t_k n \rfloor}  \mathbb{E}\left[  \overline{X}_i(\lfloor t_k n \rfloor, \varepsilon_m)^{2 \alpha} \right],
\end{align*}
which tends to zero by the same arguments  as in the proof of Lemma \ref{expectations second moment}. Hence $ V_{n,m}^3 $ tends to zero as well as $ n $ and $ m $ tend to infinity. Last, the random variables
 $ \overline{\chi}_i( \lfloor t_j n \rfloor, \varepsilon)  $ for $ i \in \lbrace1,...,\lfloor t_j n \rfloor \rbrace $, defined in Equation \ref{chi}, are independent random variables and $ \underline{\chi}_i( \lfloor t_j n \rfloor, \varepsilon)  $ for $ i \in \lbrace1,...,\lfloor t_j n \rfloor\rbrace $ are independent random variables with  
\[\lim_{m \rightarrow \infty}  \overline{\chi}_i( \lfloor t_j n \rfloor, \varepsilon_m)  = \lim_{m \rightarrow \infty}   \underline{\chi}_i( \lfloor t_j n \rfloor, \varepsilon_m) \quad a.s. \]
By the  inequalities (\ref{bounds}) we conclude that $ \lim_{m \rightarrow \infty} \lim_{n \rightarrow \infty} V_{n,m}^2=0$. This proves Equation (\ref{Variance}).
\end{proof}

\subsection{Representation as stable integral}

In this section, we aim to give a characterization of the Shark Random Swim in the subcritical case in terms stable integrals.  Recall that in Section \ref{erdb}, we introduced a Yule process with mutation $ Y_i(t) $ for $ i \in  \mathbb{N} $, and we denote by $ (\Omega, \mathcal{F}, \mathbb{P}) $ the underlying probability space. We follow the definition of stable integrals introduced in Section \ref{sttabel}. Consider the measure space \[( (0,\infty) \times \Omega ,  \mathcal{B} \times \mathcal{F}, \lambda \times \mathbb{P} ), \] where $ \mathcal{B} $ denotes the Borel $ \sigma $-algebra and $ \lambda $  the Lebesgue measure. As in the previous section, we consider the map  $ f:(0, \infty) \times \Omega \rightarrow \mathbb{R}_+ $ with
\begin{align*}
f(x,\omega):= Y_1(-\log(x)+\log(1-p))(\omega)1_{\lbrace x < (1-p) \rbrace }.
\end{align*}
We then have  for each $ t>0 $ that
\[f\left(\frac{x}{t},\omega\right)\in L^{\alpha}\left((0,\infty) \times \Omega ,  \mathcal{B} \times \mathcal{F}, \lambda \times \mathbb{P} \right), \]
since by Tonelli's theorem
\[\int_{(0,\infty) \times \Omega} \bigg\vert f \left(\frac{x}{t},\omega \right) \bigg\vert^\alpha dx \times d\mathbb{P}(\omega) \leq \int_{(0,1)} \left( \frac{t}{y}\right)^{\alpha p} dy < \infty.  \]
Hence by Lemma \ref{characteristic}, there exists a exists a stochastic process $ S(t) $  whose finite-dimensional distributions $ (S(t_1),...S(t_k)) $, with $ k \in \mathbb{N}  $, have characteristic function
\begin{equation*} 
 \exp\left(  \int_{(0,\infty) \times \Omega} \bigg\vert \sum_{j=1}^k \theta_jf\left(\frac{x}{t_j}, \omega \right)\bigg\vert^\alpha dx \times d\mathbb{P}(\omega) \right).
 \end{equation*} Since we have shown in the previous section that
\begin{align*}
\lim_{n \rightarrow \infty} & \exp\left( i \sum_{j=1}^k  \theta_j \left( \frac{1}{n}\right)^{\frac{1}{\alpha}}  S_{\lfloor t_j n \rfloor }\right) \\ &= \exp\left(  \int_{(0,\infty) \times \Omega} \bigg\vert \sum_{j=1}^k \theta_jf\left(\frac{x}{t_j}, \omega \right)\bigg\vert^\alpha dx \times d\mathbb{P}(\omega) \right),
\end{align*}
we conclude that the finite dimensional distributions of the Shark Random Swim converge to the finite dimensional distributions of the stochastic process 
\[S(t)=\int_{(0,\infty) \times \Omega} f\left(\frac{x}{t}, \omega \right)dM(x,\omega), \quad  t \in \mathbb{R}_+ \]
where $ M $ is an $ \alpha $-stable random measure with control measure $ dx \times d\mathbb{P}( \omega)  $.  The $ \alpha $-stable random measure contains the two layers of randomness appearing in the Shark Random Swim. The first part, $ dx $,  of its control measure is the contribution from the stable random variables and the second part, $ \mathbb{P}( \omega) $, is coming from the cluster sizes.

 We now aim to give a better understanding of the limiting process. In the case $ \alpha = 2 $, the limiting process is Gaussian and can be characterized by its covariance functions given in the next statement. 
\begin{lemma} \label{covariance2}
Let $ \alpha = 2 $ and let $ 0 <r < t $. We then have
\[\textup{Cov}(S(r),S(t))=2 \cdot \left(\frac{t}{r}\right)^{p} \frac{r}{1-2p}.\]
\end{lemma}

\begin{proof}
By Lemma \ref{covariance} and Tonelli's theorem we have  
\begin{align*}
\textup{Cov}(S(r),S(t)) &=2 \int_{(0,\infty) \times \Omega} f\left(\frac{x}{r},\omega \right) f\left(\frac{x}{t},\omega \right) dx \times d\mathbb{P}(\omega)
\\ &= 2 \int_{0}^{(1-p)r} \mathbb{E}\left[ Y_1\left(\log\left( \frac{r(1-p)}{x}\right)\right) Y_1\left(\log\left( \frac{t(1-p)}{x}\right)\right)\right] dx.
\end{align*}
Now, recall that $ (Y_1(t))_{t \geq 0} $ is a Yule process with birth rate $ p $, and let by $ \mathbb{P}_k $ denote the law of $ Y_1(t) $ when $ Y_1(0)=k $. For $ s < t $ we then have by the Markov property and the branching  property
\begin{align*}
\mathbb{E}_1[Y_1(s)Y_1(t)]&=\mathbb{E}_1\left[\mathbb{E}_1[Y_1(s)Y_1(t)\vert \mathcal{F}_s]\right]
\\ &=  \mathbb{E}_1 \left[ Y_1(s) \mathbb{E}_{Y_1(s)}\left[Y_1(t-s)\right]\right] \\ &=   \mathbb{E}_1[Y_1(t-s)] \cdot \mathbb{E}_1\left[Y_1(s)^2\right] \\ &= e^{-p(t-s)} \cdot \left( \frac{2-e^{-ps}}{e^{-2ps}} \right),
\end{align*}
where we used in the last inequality that for each $ t >0 $, the random variable $ Y_1(t) $ has the Geometric distribution with parameter $ e^{-tp} $. Hence
\begin{align*}
\textup{Cov}(S(r),S(t)) &=2 \left( \frac{t}{r}\right)^p \int_{0}^{(1-p)r} \left( 2 \left( \frac{r(1-p)}{x} \right)^{2p} -   \left(\frac{r(1-p)}{x} \right)^{p}\right)dx \\&= 2 \left(\frac{t}{r}\right)^{p} \frac{r}{1-2p}.
\end{align*}
\end{proof}

\begin{remark}
If we set $ p=2q-1 $, with $ q $ being the memory parameter of the first formulation of the ERW, we recover twice the covariance functions of the limiting process of the ERW in the subcritical case, see Baur and Bertoin \cite{Baur}.  Recall that if $ \alpha = 2 $, the steps of the Shark Random Swim  are $ \mathcal{N}(0,2) $ distributed, which explains the factor $ 2 $.
\end{remark}
\subsection{Critical case}
In this section we aim to show the following result.
\begin{theorem}
Let $ k \in \mathbb{N} $ and $0 < t_1< t_2<...<t_k \in \mathbb{R} $.  We have the distributional convergence
\[\left(\frac{1}{n^{t_1} \log(n)} S_{\lfloor n^{t_1} \rfloor},...,\frac{1}{n^{t_k} \log(n)} S_{\lfloor n^{t_k} \rfloor} \right) \rightarrow (\tilde{S}(t_1),...,\tilde{S}(t_k) ,\]
where the random vector $ \left(\tilde{S}(t_1),...,\tilde{S}(t_k) \right) $ has joint characteristic function
\begin{align*}
\exp\left(i  \sum_{m=1}^k a_m \tilde{S}(t_m) \right) =
 \exp\left( -(1-p)\Gamma(\alpha + 1) \sum_{m=1}^k (t_m - t_{m-1}) \bigg\vert \sum_{j=m}^k a_j \bigg\vert^\alpha \right),
 \end{align*}
with $ a_1,...,a_k \in \mathbb{R} $, and we recognize  the characteristic function of the finite dimensional distributions of an $ \alpha $-stable Lévy process.
\end{theorem}
Recall that $ \mathcal{F}_n: = \sigma(\vert c_{1,n} \vert,...,\vert c_{n,n} \vert) $. We have
\begin{align*}
\mathbb{E}&\left[\exp\left( i \sum_{j=1}^k a_j \left( \frac{1}{n^{t_j}\log(n)}\right)^{\frac{1}{\alpha}}  S_{\lfloor n^{t_j}  \rfloor } \right) \bigg\vert \mathcal{F}_{\lfloor n^{t_k}  \rfloor} \right] \\&= \exp\left(-\frac{1}{\log(n)}\sum_{m=1}^\infty \bigg\vert \sum_{j=1}^k a_j \cdot \frac{1}{n^{t_j}} \vert c_{m,\lfloor n^{t_j}  \rfloor } \vert 1_{\lbrace m \leq \lfloor n^{t_j}  \rfloor  \rbrace } \bigg\vert^\alpha \right). 
\end{align*}
As in the previous section it will be enough to show the convergence in probability
\begin{align*}  \lim_{m \rightarrow \infty} \lim_{n \rightarrow \infty} \sum_{i=1}^\infty & \bigg\vert \sum_{j=1}^k a_j \cdot Y_{i }(T(\lfloor n^{t_j}  \rfloor) \vert 1_{\lbrace i \leq \lfloor t_j n \rfloor  \rbrace } \bigg\vert^\alpha 1_{G_{n,m}} \\ &= \exp\left( -(1-p)\Gamma(\alpha + 1) \sum_{l=1}^k (t_l - t_{l-1}) \bigg\vert \sum_{j=l}^k a_j \bigg\vert^\alpha \right),
\end{align*}
where the sequence of events $ G_{n,m} := E_m(\lfloor n^{\delta t_k} \rfloor, \lfloor n^{ t_k} \rfloor)$ with  $ 0< \delta < \min\left(1,\frac{1}{t_k} \right) $ is defined as in Lemma \ref{sequuence} and thus \[\lim_{m \rightarrow \infty} \lim_{n \rightarrow \infty}\mathbb{P}(G_{n,m})=1 .\] 

For the proof we further need to construct a second sequence of events. Recall that  $ c_{i,n}^\prime $ denote the subtree rooted at node $ i \in \lbrace 1,...,n \rbrace$ after percolation and that
 \begin{equation*}
  \vert c_{i,n}^{\prime} \vert \overset{d}{=} \vert c_{1,Y(n,i)} \vert,
  \end{equation*}
  where  $Y(n,i) $ is Beta-binomial distributed with parameter $ (n-i,1,i-1) $ and independent of the cluster sizes. Moreover \[\lim_{n \rightarrow \infty} \frac{1}{n} Y(n,i) =  Y(i) \quad a.s.\] and $ Y(i) $ is Beta distributed with parameters $ 1 $ and $ i-1 $. We then have:
\begin{lemma}
\label{sett}
Let $0< \delta < 1$ and $ \mu >0 $. Let $ i \in \lbrace 1,...,n \rbrace $ and define the events  
\[B_i(\mu):= \left\lbrace \bigg\vert \frac{1}{n} Y(n,i)- Y(i) \bigg\vert<\mu \right\rbrace \quad \textnormal{and} \quad F_n(\delta, \mu) :=\bigcap_{i=\lfloor n^{\delta} \rfloor}^n B_i(\mu). \] We then have $ \lim_{n \rightarrow}\mathbb{P}(F_n(\delta, \mu))=1 $.
\end{lemma}
\begin{proof} By Markov's inequality we have that
\begin{align*}
  \mathbb{P}(F_n(\delta, \mu)) &= 1-\sum_{i=\lfloor n^\delta \rfloor}^n \mathbb{P}(B_i(\mu)^c) \\ &\geq 1-\frac{1}{\varepsilon^2} \sum_{i=\lfloor n^\delta \rfloor}^n \mathbb{E}\left[ \left( \frac{1}{n} Y(n,i)-  Y(i) \right)^2 \right].
  \end{align*}
  Moreover by Lemma \ref{freedman} we have
  \begin{align*}
\sum_{i = \lfloor n^\delta \rfloor}^n  & \mathbb{E}\left[\left( \frac{1}{n} Y(n,i)-  Y(i)\right)^2\right] \\ &= \sum_{i = \lfloor n^\delta \rfloor}^n  \mathbb{E}\left[ \mathbb{E}\left[\left(  \frac{1}{n} Y(n,i)- Y(i)\right)^2 \bigg\vert Y(i)\right]\right] \\ &= \sum_{i = \lfloor n^\delta \rfloor}^n   \frac{1}{n}   \int_0^1 (x-x^ 2)\cdot (1-x)^ {i-2} \frac{1}{\beta(1,i-1)} dx \\ &= \sum_{i = \lfloor n^\delta \rfloor}^n  \frac{1}{n}\cdot \frac{i-1}{i+i^2}, 
  \end{align*}
 which tends to zero as $ n $ tends to infinity, and we have  shown Lemma \ref{sett}.
\end{proof}
Now let us define for $ 0 < \varepsilon < 1 $ the event
\begin{equation}
\label{Kistli}
H(n, \varepsilon):=\left\lbrace \left\vert \frac{1}{n} \vert c_{1,n} \vert - X_1 \right\vert \leq \varepsilon \right\rbrace, 
\end{equation}
 let $ (\mu_m)_{m \in \mathbb{N}} $ be a  sequence with $ 0 < \mu_n < 1 $ and $ \lim_{m \rightarrow \infty} \mu_m =0 $, and let $ 0<\delta<\min(\frac{1}{t_k},1) $. We then define the event
\begin{equation} \label{eve} E_{n,m} := H\left(\lfloor n^{t_1}  \cdot \mu_m \cdot (1-\varepsilon) \rfloor, \varepsilon\right) \cap F_{\lfloor n^{t_k} \rfloor}(\delta, \varepsilon), 
\end{equation}
and note that
\[\lim_{m \rightarrow \infty} \lim_{n \rightarrow \infty} \mathbb{P}(E_{n,m})=1.\]
We are now ready to tackle the proof of Theorem \ref{theorem 2} which is, by dominated convergence, a direct of consequence of the following result.
\begin{proposition}
Let $ p < 1 $ and $ \alpha p = 1 $. We then have the convergence in probability
\begin{align*}
\lim_{m \rightarrow \infty} \lim_{n \rightarrow \infty} & \frac{1}{\log(n)} \sum_{i=1}^\infty \bigg\vert \sum_{j=1}^k \left( \frac{1}{ n^{t_j} } \right)^{\frac{1}{\alpha}} a_j  Y_{i}(T(\lfloor n^{t_j } \rfloor )) 1_{\lbrace i \leq \lfloor n^{t_j } \rfloor \rbrace}  \bigg\vert^\alpha 1_{G_{n,m}} \\ &= (1-p)\Gamma(\alpha + 1) \sum_{l=1}^k (t_l - t_{l-1}) \bigg\vert \sum_{j=l}^k a_j \bigg\vert^\alpha .
\end{align*}
\end{proposition}
\begin{proof}
We shall use the second moment method. First we show that
\begin{align}
\label{exp}
\lim_{m \rightarrow \infty} \lim_{n \rightarrow \infty}  \frac{1}{\log(n)} \sum_{i=1}^\infty & \mathbb{E} \left[ \bigg\vert \sum_{j=1}^k  \left( \frac{1}{ n^{t_j} } \right)^{\frac{1}{\alpha}} a_j Y_i(T( \lfloor n^{t_j } \rfloor))  1_{\lbrace i \leq \lfloor n^{t_j } \rfloor \rbrace}  \bigg\vert^\alpha 1_{G_{n,m}}\right] \nonumber \\ &= (1-p)\Gamma(\alpha + 1) \sum_{l=1}^k (t_l - t_{l-1}) \bigg\vert \sum_{j=l}^k a_j \bigg\vert^\alpha .
\end{align}
If $ \alpha = 2 $, the result is a simple computation using the bound developed in Equation (\ref{bounds}), and the first and second moments of Geometric distributed random variables. We let thus $ \alpha <2 $ and aim to show that
\begin{align}
\label{exp2}
\lim_{m \rightarrow \infty} \lim_{n \rightarrow \infty}  \frac{1}{\log(n)} \sum_{i=1}^\infty & \mathbb{E} \left[ \bigg\vert \sum_{j=1}^k  \left( \frac{1}{ n^{t_j} } \right)^{\frac{1}{\alpha}} a_j \vert c_{i, \lfloor n^{t_j } \rfloor}  \vert 1_{\lbrace i \leq \lfloor n^{t_j } \rfloor \rbrace}  \bigg\vert^\alpha 1_{
E_{n,m}}\right] \nonumber \\ &= (1-p)\Gamma(\alpha + 1) \sum_{l=1}^k (t_l - t_{l-1}) \bigg\vert \sum_{j=l}^k a_j \bigg\vert^\alpha ,
\end{align}
with $ E_{n,m} $ defined as in Equation (\ref{eve}). We first justify that this is equivalent to Equation (\ref{exp}). Indeed,

\begin{align*}
\mathbb{E} & \left[ \bigg\vert \sum_{j=1}^k  \left( \frac{1}{ n^{t_j} } \right)^{\frac{1}{\alpha}} a_j  \vert c_{i, \lfloor n^{t_j } \rfloor}  \vert 1_{\lbrace i \leq \lfloor n^{t_j } \rfloor \rbrace}  \bigg\vert^\alpha 1_{\lbrace \Omega \backslash E_{n,m} \rbrace}\right] \\ &\leq \sum_{j=1}^k \frac{1}{ n^{t_j} } \vert a_j  \vert^\alpha \mathbb{E} \left[ \vert    c_{i, \lfloor n^{t_j } \rfloor}   \vert^2 \right]^{\frac{\alpha}{2}} \mathbb{P}(\Omega \backslash E_{n,m})^{\frac{\alpha}{2}} \\ &\leq \sum_{j=1}^k c(2,p) \cdot \vert a_j  \vert^\alpha \frac{1}{i} \cdot \mathbb{P}(\Omega \backslash E_{n,m})^{\frac{\alpha}{2}},
\end{align*}
where the constant $ c(2,p) $ is as in Lemma \ref{boundibound}, and we conclude that 
\begin{align*}
\lim_{m \rightarrow \infty} \lim_{n \rightarrow \infty}  \frac{1}{\log(n)} \sum_{i=1}^\infty & \mathbb{E} \left[ \bigg\vert \sum_{j=1}^k  \left( \frac{1}{ n^{t_j} } \right)^{\frac{1}{\alpha}} a_j \vert c_{i, \lfloor n^{t_j } \rfloor}  \vert 1_{\lbrace i \leq \lfloor n^{t_j } \rfloor \rbrace}  \bigg\vert^\alpha 1_{\lbrace \Omega \backslash E_{n,m} \rbrace}\right]=0.
\end{align*}
In the same spirit one can show that
\begin{align*}
\lim_{m \rightarrow \infty} \lim_{n \rightarrow \infty}  \frac{1}{\log(n)} \sum_{i=1}^\infty & \mathbb{E} \left[ \bigg\vert \sum_{j=1}^k  \left( \frac{1}{ n^{t_j} } \right)^{\frac{1}{\alpha}} a_j Y_{i} (T( \lfloor n^{t_j } \rfloor))  1_{\lbrace i \leq \lfloor n^{t_j } \rfloor \rbrace}  \bigg\vert^\alpha 1_{\lbrace \Omega \backslash G_{n,m} \rbrace}\right]=0,
\end{align*}
which proves the equivalence of Equations (\ref{exp}) and (\ref{exp2}). As in the previous section we now split the sum in two parts. We start by establishing a bound for the first $ \lfloor n^ {\delta t_k }  \rfloor $ terms.
\begin{align*}
 \sum_{i=1}^{\lfloor n^{\delta t_k}  \rfloor} & \mathbb{E} \left[ \bigg\vert \sum_{j=1}^k  \left( \frac{1}{ n^{t_j} } \right)^{\frac{1}{\alpha}} a_j  \vert c_{i, \lfloor n^{t_j } \rfloor} \vert 1_{\lbrace i \leq \lfloor n^{t_j } \rfloor \rbrace}  \bigg\vert^\alpha 1_{E_{n,m}}\right]  \\ &\leq  \sum_{i=1}^{\lfloor n^{\delta t_k}  \rfloor} \sum_{j=1}^k \frac{1}{ n^{t_j} } \vert a_j \vert^\alpha \mathbb{E} \left[ \vert c_{i, \lfloor n^{t_j } \rfloor}  \vert^\alpha 1_{E_{n,m}}   \right] 1_{\lbrace i \leq \lfloor n^{t_j } \rfloor \rbrace} \\ &\leq c_1(\alpha,p) \sum_{i=1}^{\lfloor n^{\delta t_k}  \rfloor} \sum_{j=1}^k \vert a_j \vert^\alpha \cdot   \frac{1}{i}.
\end{align*}
where the  constant $ c_1(\alpha,p) $ is as in Lemma  \ref{boundibound}, hence:
\begin{align}
\label{v}
\lim_{n \rightarrow \infty}  \frac{1}{\log(n)}  \sum_{i=1}^{\lfloor n^{\delta t_k}  \rfloor} & \mathbb{E} \left[ \bigg\vert \sum_{j=1}^k  \left( \frac{1}{ n^{t_j} } \right)^{\frac{1}{\alpha}} a_j  \vert c_{i, \lfloor n^{t_j } \rfloor} \vert 1_{\lbrace i \leq \lfloor n^{t_j } \rfloor \rbrace}  \bigg\vert^\alpha 1_{E_{n,m}}\right]  \nonumber \\ &\leq c_1(\alpha,p) \sum_{j=1}^k \vert a_j \vert^\alpha \cdot  \delta t_k.
\end{align}
For the other part of the sum, note that
\begin{align*}
\sum_{i= \lfloor n^{\delta t_k } \rfloor }^\infty & \mathbb{E} \left[ \bigg\vert \sum_{j=1}^k \left( \frac{1}{ n^{t_j} } \right)^{\frac{1}{\alpha}} a_j  \vert c_{i, \lfloor n^{t_j } \rfloor} \vert 1_{\lbrace i \leq \lfloor n^{t_j } \rfloor \rbrace}  \bigg\vert^\alpha 1_{E_{n,m}} \right] \\ &= (1-p) \sum_{i= \lfloor n^{\delta t_k } \rfloor }^\infty \mathbb{E} \left[ \bigg\vert \sum_{j=1}^k  \left(\frac{1}{ n^{t_j} } \right)^{\frac{1}{\alpha}}a_j  \vert c^\prime_{i, \lfloor n^{t_j } \rfloor} \vert 1_{\lbrace i \leq \lfloor n^{t_j } \rfloor \rbrace}  \bigg\vert^\alpha 1_{E_{n,m}} \right]  \\&= (1-p) \sum_{i= \lfloor n^{\delta t_k } \rfloor  }^\infty \mathbb{E}  \left[ \bigg\vert \sum_{j=1}^k \left( \frac{1}{n^{t_j}} \right)^{\frac{1}{\alpha}} a_j \vert c_{1,Y(\lfloor n^{t_j } \rfloor,i)} \vert 1_{\lbrace i \leq \lfloor n^{t_j } \rfloor \rbrace}  \bigg\vert^\alpha 1_{E_{n,m}} \right]. 
\end{align*}
Now, define the event
\[A_{ij}:=\left\lbrace \omega \in \Omega: \sgn(a_j)=\sgn\left(\sum_{j=1}^k  a_j \vert c_{1,Y(\lfloor n^{t_j } \rfloor,i)} \vert (\omega) 1_{\lbrace i \leq \lfloor n^{t_j } \rfloor \rbrace}  \right) \right\rbrace,\]
and, for $ \varepsilon > 0 $, the map $ f^{\varepsilon}:\Omega \rightarrow \mathbb{R} $ with
\[f_{ij}^\varepsilon(\omega)=(1+\varepsilon)1_{A_{ij}}(\omega)+(1-\varepsilon)1_{A_{ij}^c}(\omega).\]
By the same reasoning as in the previous section, and the fact that on $ E_{n,m} $ we have the almost sure inequality
\[\lfloor n^{t_j } \rfloor Y(i) (1-\varepsilon) \leq  Y(\lfloor n^{t_j } \rfloor,i) \leq \lfloor n^{t_j } \rfloor Y(i) (1+\varepsilon),\]
for all $ i \in \lbrace \lfloor n^{\delta t_k}  \rfloor,..., \lfloor n^{t_k}  \rfloor\rbrace $, we then have the inequality
\begin{align}
\label{sdfg}
\mathbb{E}  & \left[ \bigg\vert \sum_{j=1}^k \left( \frac{1}{n^{t_j}} \right)^{\frac{1}{\alpha}} a_j \vert c_{1,Y(\lfloor n^{t_j } \rfloor,i)} \vert 1_{\lbrace i \leq \lfloor n^{t_j } \rfloor \rbrace}  \bigg\vert^\alpha 1_{E_{n,m}} \right] \nonumber \\ &\leq \mathbb{E}  \left[ \bigg\vert \sum_{j=1}^k \left( \frac{1}{n^{t_j}} \right)^{\frac{1}{\alpha}} a_j \vert c_{1,\left\lfloor\lfloor n^{t_j } \rfloor Y(i) f_{ij}^\varepsilon\right\rfloor} \vert 1_{\lbrace i \leq \lfloor n^{t_j } \rfloor \rbrace}  \bigg\vert^\alpha 1_{E_{n,m}} \right].
\end{align}
Recall that the random variable $ Y(i) $ is Beta distributed with parameters $ 1 $ and $ i-1 $ and thus, by conditioning on $ Y(i) $, Line (\ref{sdfg}) is equal to
\begin{align}
\label{t3}
 \int_0^1 \mathbb{E}  \left[ \bigg\vert \sum_{j=1}^k \left( \frac{1}{n^{t_j}} \right)^{\frac{1}{\alpha}} a_j \vert c_{1,\left\lfloor \lfloor n^{t_j } \rfloor x f_{ij}^\varepsilon \right\rfloor} \vert 1_{\lbrace i \leq \lfloor n^{t_j } \rfloor \rbrace}  \bigg\vert^\alpha 1_{E_{n,m}} \right] (i-1) \cdot (1-x)^{i-2}  dx.
\end{align}
We now split the integral into two parts. For the first part, note that for $ x \in (\mu_m,1) $ we have by the definition of the event $ E_{n,m} $, thanks Equations (\ref{Kistli}) and (\ref{eve}), that
\begin{align*}
\mathbb{E} &  \left[ \bigg\vert \sum_{j=1}^k \left( \frac{1}{n^{t_j}} \right)^{\frac{1}{\alpha}} a_j \vert c_{1,\left\lfloor \lfloor n^{t_j } \rfloor x f_{ij}^\varepsilon \right\rfloor} \vert 1_{\lbrace i \leq \lfloor n^{t_j } \rfloor \rbrace}  \bigg\vert^\alpha 1_{E_{n,m}} \right] \\ &\leq \mathbb{E}   \left[ \bigg\vert \sum_{j=1}^k \left(\frac{1}{n^{t_j}} \right)^{\frac{1}{\alpha}} a_j X_1 \cdot \left(\lfloor  n^{t_j} \rfloor x f_{ij}^\varepsilon \right)^{\frac{1}{\alpha}} \cdot  f_{ij}^\varepsilon   \bigg\vert^\alpha   1_{ \lbrace i \leq \lfloor n^{t_j} \rfloor \rbrace } 1_{E_{n,n}}  \right].
\end{align*}
We combine this inequality with Equation (\ref{t3}). We  have
\begin{align} 
\label{t1}
  \int_{\mu_m}^1  \mathbb{E}  & \left[ \bigg\vert \sum_{j=1}^k \left(\frac{1}{n^{t_j}} \right)^{\frac{1}{\alpha}} a_j X_1 \cdot \left(\lfloor  n^{t_j} \rfloor x f_{ij}^\varepsilon \right)^{\frac{1}{\alpha}} \cdot  f_{ij}^\varepsilon   \bigg\vert^\alpha   1_{ \lbrace i \leq \lfloor n^{t_j} \rfloor \rbrace } 1_{E_{n,n}}  \right] (i-1)  (1-x)^{i-2}dx \nonumber \\ &\leq \mathbb{E}[\vert X_1 \vert^\alpha] \cdot  \mathbb{E}\left[\bigg\vert \sum_{j=1}^k  (f_{ij}^{\varepsilon})^{1+\frac{1}{\alpha}} \cdot a_j    \cdot  1_{ \lbrace i \leq \lfloor n^{t_j} \rfloor \rbrace } \bigg\vert^\alpha \right] \int_{\mu_m}^1 x(i-1)  (1-x)^{i-2} dx  \nonumber \\ &\leq \mathbb{E}[\vert X_1 \vert^\alpha ] \cdot  \mathbb{E}\left[\bigg\vert \sum_{j=1}^k  (f_{ij}^{\varepsilon})^{1+\frac{1}{\alpha}} \cdot a_j  \cdot  1_{ \lbrace i \leq \lfloor n^{t_j} \rfloor \rbrace } \bigg\vert^\alpha  \right] \int_{0}^1 x (i-1)  (1-x)^{i-2} dx \nonumber \\ &=  \Gamma(\alpha +1) \cdot  \mathbb{E}\left[\bigg\vert \sum_{j=1}^k  (f_{ij}^{\varepsilon})^{1+\frac{1}{\alpha}} \cdot a_j   \cdot  1_{ \lbrace i \leq \lfloor n^{t_j} \rfloor \rbrace } \bigg\vert^\alpha   \right] \frac{1}{i} .
\end{align}
For the other part of the integral, we have
\begin{align*}
\int_0^{\mu_m} & \sum_{j=1}^k  \frac{1}{n^{t_j}} \cdot \vert a_j \vert^\alpha \mathbb{E}\left[\left\vert c_{1,\left\lfloor \lfloor n^{t_j } \rfloor x f_{ij}^\varepsilon \right\rfloor} \right\vert^\alpha 1_{E_{n,m} } \right] (i-1)  (1-x)^{i-2} dx \\ &\leq \int_0^{\mu_m} \sum_{j=1}^k  \frac{1}{n^{t_j}} \cdot \vert a_j \vert^\alpha \mathbb{E}\left[\left\vert c_{1,\left\lfloor \lfloor n^{t_j } \rfloor x f_{ij}^\varepsilon \right\rfloor} \right\vert^2 \right]^{\frac{\alpha}{ 2}} (i-1) (1-x)^{i-2} dx \\ &\leq \int_0^{\mu_m} \sum_{j=1}^k    \vert a_j \vert^\alpha \left( \left(\frac{1}{n^{t_j}}\right)^{2p} \mathbb{E}\left[\left\vert c_{1,\left\lfloor \lfloor n^{t_j } \rfloor x f_{ij}^\varepsilon \right\rfloor} \right\vert^2 \right]\right)^{\frac{\alpha}{ 2}} (i-1)(1-x)^{i-2} dx .
\end{align*}
Now let $ c(p):=\sup_{n>0}  \frac{1}{n^{2 p} } \cdot \frac{1}{p} \cdot \frac{\Gamma(2p+n)}{\Gamma(2p)\Gamma(n)} < \infty $. By Lemma \ref{Mittag leffler2}, we then have for all $ x \in (0,1) $ that
\[\left(\left(\frac{1}{n^{t_j}}\right)^{2p} \mathbb{E}\left[\left\vert c_{1,\left\lfloor \lfloor n^{t_j } \rfloor x f_{ij}^\varepsilon \right\rfloor} \right\vert^2 \right]\right)^{\frac{\alpha}{ 2}} \leq x \cdot (1+\varepsilon) \cdot c(p). \]
Hence 
 \begin{align*}
\int_0^{\mu_m} \mathbb{E}  & \left[ \bigg\vert \sum_{j=1}^k \left( \frac{1}{n^{t_j}} \right)^{\frac{1}{\alpha}} a_j \vert c_{1,\left\lfloor \lfloor n^{t_j } \rfloor x f_{ij}^\varepsilon \right\rfloor} \vert 1_{\lbrace i \leq \lfloor n^{t_j } \rfloor \rbrace}  \bigg\vert^\alpha 1_{E_{n,m}} \right]  (i-1) (1-x)^{i-2}  dx \\ &\leq \int_0^{\mu_m} \sum_{j=1}^k    \vert a_j \vert^\alpha \cdot c(p) \cdot (1+\varepsilon)  \cdot x (i-1)  (1-x)^{i-2} dx.
\end{align*}
Now note that for all $ i > 4 $ we have on  $ (0, 1) $ that $ x \cdot (i-1) \cdot (1-x)^{i-2}  \leq  \frac{2}{i-1} $. We hence arrive at
 \begin{align}
 \label{t6}
\int_0^{\mu_m} \mathbb{E}  & \left[ \bigg\vert \sum_{j=1}^k \left( \frac{1}{n^{t_j}} \right)^{\frac{1}{\alpha}} a_j \vert c_{1,\left\lfloor \lfloor n^{t_j } \rfloor x f_{ij}^\varepsilon \right\rfloor} \vert 1_{\lbrace i \leq \lfloor n^{t_j } \rfloor \rbrace}  \bigg\vert^\alpha 1_{E_{n,m} } \right] ( i-1) \cdot (1-x)^{i-2}  dx \nonumber \\ &\leq \sum_{j=1}^k    \vert a_j \vert^\alpha c(p) \cdot (1+\varepsilon) \cdot \frac{2}{i-1} \cdot \mu_m.
\end{align}
Combining the Inequalities  (\ref{t6}) and (\ref{t1}) we arrive at
\begin{align*}
\mathbb{E}  & \left[ \bigg\vert \sum_{j=1}^k \left( \frac{1}{n^{t_j}} \right)^{\frac{1}{\alpha}} a_j \vert c_{1,Y(\lfloor n^{t_j } \rfloor,i)} \vert 1_{\lbrace i \leq \lfloor n^{t_j } \rfloor \rbrace}  \bigg\vert^\alpha 1_{E_{n,m}} \right] \\ &\leq (1-p)\Gamma(\alpha +1) \mathbb{E}\left[\bigg\vert \sum_{j=1}^k  (f_{ij}^{\varepsilon})^{1+\frac{1}{\alpha}} \cdot a_j   \bigg\vert^\alpha \right] \frac{1}{i} \cdot 1_{ \lbrace i \leq \lfloor n^{t_j} \rfloor\rbrace}\\  &\textnormal{\textcolor{white}{.........}}+(1-p)\sum_{j=1}^k    \vert a_j \vert^\alpha c(p) \cdot (1+\varepsilon) \cdot \frac{2}{i-1} \cdot \mu_m 
\end{align*}
Taking limits we deduce that 
\begin{align*}
\lim_{m \rightarrow \infty} & \lim_{n \rightarrow \infty}  \frac{1}{\log(n)} \sum_{i=\lfloor n^{\delta t_k } \rfloor}^\infty \bigg\vert \sum_{j=1}^k \left( \frac{1}{ n^{t_j} } \right)^{\frac{1}{\alpha}} a_j  \vert c_{i, \lfloor n^{t_j } \rfloor} \vert 1_{\lbrace i \leq \lfloor n^{t_j } \rfloor \rbrace}  \bigg\vert^\alpha 1_{E_{n,m}} \nonumber \\  &=  (1-p)\Gamma(\alpha+1)\sum_{l=1}^k (t_l - t_{l-1}) \mathbb{E}\left[\bigg\vert \sum_{j=l}^k  (f_{ij}^{\varepsilon})^{1+\frac{1}{\alpha}} \cdot a_j   \cdot   \bigg\vert^\alpha \right](1-\delta t_k),
\end{align*}
and  by letting $ \varepsilon $ tend to zero we arrive at
\begin{align}
\label{u}
\lim_{m \rightarrow \infty} & \lim_{n \rightarrow \infty}  \frac{1}{\log(n)} \sum_{i=\lfloor n^{\delta t_k } \rfloor}^\infty \bigg\vert \sum_{j=1}^k \left( \frac{1}{ n^{t_j} } \right)^{\frac{1}{\alpha}} a_j  \vert c_{i, \lfloor n^{t_j } \rfloor} \vert 1_{\lbrace i \leq \lfloor n^{t_j } \rfloor \rbrace}  \bigg\vert^\alpha 1_{E_{n,m}} \nonumber \\  &=  (1-p)\Gamma(\alpha+1)\sum_{l=1}^k (t_l - t_{l-1}) \sum_{j=l}^k   \vert a_j   \vert^\alpha(1-\delta t_k), 
\end{align}
and we conclude by combining Inequalities (\ref{v}) and (\ref{u}) and letting $ \delta $ tend to zero. The lower bound can be shown in the same spirit.
We are left to show that 
\begin{align*}
  \textnormal{Var} \left(\frac{1}{\log(n)} \sum_{i=\lfloor n^{\delta t_k} \rfloor}^\infty \bigg\vert \sum_{j=1}^k \left( \frac{1}{ n^{t_j} } \right)^{\frac{1}{\alpha}} a_j Y_i(T(\lfloor n^{t_j } \rfloor)) \vert 1_{\lbrace i \leq \lfloor n^{t_j } \rfloor \rbrace}  \bigg\vert^\alpha 1_{G_{n,m}} \right)
\end{align*}
tends to zero, as $ n $ and $ m $ tend to infinity.
Again, we split the variance in three parts. Define
\begin{align*}
V_{n,m}^1 :&=  \sum_{i= \lfloor n^{\delta t_k} \rfloor }^\infty  \mathbb{E}\left[\bigg\vert \sum_{j=1}^k \left(\frac{1}{n^ {t_j}}\right)^{\frac{1}{\alpha}}a_j Y_{i}(T(\lfloor  n^{t_j} \rfloor)) 1_{ \lbrace  i \leq  \lfloor  n^{t_j} \rfloor\rbrace } \bigg\vert^{2 \alpha} 1_{G_{n,m}}\right], 
\\ V_{n,m}^2 :&= \frac{2}{n^2} \cdot \sum_{i \neq l} \textup{Cov}\bigg(\bigg\vert \sum_{j=1}^k \left(\frac{1}{n^ {t_j}}\right)^{\frac{1}{\alpha}}a_j  Y_{l}(T(\lfloor n^{t_j}\rfloor)) 1_{ \lbrace  l \leq  \lfloor  n^{t_j} \rfloor\rbrace } \bigg\vert^{\alpha} 1_{G_{n,m}}, \\ &\textnormal{\textcolor{white}{...........................}} \bigg\vert \sum_{j=1}^k \left(\frac{1}{n^ {t_j}}\right)^{\frac{1}{\alpha}}a_j  Y_{i} (T(\lfloor n^{t_j} \rfloor)) 1_{ \lbrace  i \leq  \lfloor  n^{t_j} \rfloor\rbrace } \bigg\vert^{ \alpha} 1_{G_{n,m}}\bigg), \\
\\ V_{n,m}^ 3 :&=   \sum_{i=  \lfloor n^{\delta t_k} \rfloor }^\infty  \mathbb{E}\left[\bigg\vert \sum_{j=1}^k \left(\frac{1}{n^ {t_j}}\right)^{\frac{1}{\alpha}}a_j Y_{i}( T(\lfloor n^{t_j} \rfloor)) 1_{ \lbrace  i \leq  \lfloor  n^{t_j} \rfloor\rbrace } \bigg\vert^{ \alpha} 1_{G_{n,m}}\right]^2.
\end{align*}
By the same computations as in the proof of Lemma \ref{critical one dimensional}, we have
\begin{eqnarray*}\lim_{m \rightarrow \infty} \lim_{n \rightarrow \infty} V_{n,m}^1 \leq \lim_{m \rightarrow \infty} \lim_{n \rightarrow \infty} c_m \cdot \frac{ (t_k \cdot k \cdot a)^{2\alpha}   }{\log(n)^2} \sum_{i=1}^{\lfloor t_k n \rfloor}\left( \frac{1}{i} \right)^{2 }=0, \end{eqnarray*}
where $ c_m:=150^{\frac{\alpha}{2}}((1-p)(1+ \varepsilon_m))^{2 } $ and hence $ \lim_{m \rightarrow \infty} \lim_{n \rightarrow \infty} V_{n,m}^3  $. Last $ V_{n,m}^2 $ tends to zero as $ n $ and $ m  $ tend to infinity by the same reasoning as in  the proof of Proposition \ref{Prop1}.
\end{proof}

\section*{Acknowledgement}
I would like to thank Jean Bertoin for introducing me to this topic and for his advice and support. I would also like to thank two anonymous referees for their careful reading of an earlier version of this work and their helpful comments.
\bibliographystyle{abbrv} 
\bibliography{literature}

\end{document}